\documentclass[11pt,twoside]{article}

\usepackage[left=3cm, right=3.5cm, top=3cm]{geometry}
\usepackage{amssymb}
\usepackage{amsmath}
\usepackage{amsthm}
\usepackage{mathtools}
\usepackage{listings}
\usepackage{epsfig}
\usepackage{latexsym}
\usepackage{mathtools}
\usepackage{booktabs}
\usepackage{multirow}
\usepackage{graphicx}
\usepackage{xcolor}
\usepackage{epsfig}
\usepackage{url}
\usepackage{amsfonts}

\usepackage{hyperref}
\usepackage{textcomp}
\usepackage{indentfirst}
\usepackage{tipa}
\usepackage{dsfont}
\usepackage[mathscr]{euscript}
\usepackage{enumitem}
\usepackage{bm}
\usepackage[backend=bibtex,style=numeric,citestyle=numeric,giveninits=true,sortcites=true,maxbibnames=2]{biblatex}

\usepackage{filecontents}

\immediate\write18{bibtex \jobname}
\usepackage{color}
\definecolor{deepblue}{rgb}{0,0,0.5}
\definecolor{deepred}{rgb}{0.6,0,0}
\definecolor{deepgreen}{rgb}{0,0.5,0}

	\definecolor{DarkBlue}{rgb}{0.00,0.00,0.55}
	\definecolor{Black}{rgb}{0.00,0.00,0.00}
	\hypersetup{
		linkcolor = DarkBlue,
		anchorcolor = DarkBlue,
		citecolor = DarkBlue,
		filecolor = DarkBlue,
		colorlinks  = true,
		urlcolor = Black
	}
	
	\graphicspath{ {../Figures/} }
		\DeclareGraphicsExtensions{.eps,.pdf,.png,.jpg}  
	
\newtheorem{theorem}{Theorem}[section]
\newtheorem{lemma}[theorem]{Lemma}

\newtheorem{corollary}[theorem]{Corollary}
\theoremstyle{definition}
\newtheorem{definition}{Definition}[section]

\theoremstyle{remark}
\newtheorem{remark}{Remark}[section]

\newcommand{\TheTitle}{Efficient white noise sampling and coupling for multilevel Monte Carlo with non-nested meshes}
 
\newcommand{\TheAuthors}{M.~Croci, M.~B.~Giles, M.~E.~Rognes and P.~E.~Farrell}

\title{{\TheTitle}\thanks{\textbf{Funding:} This research is supported by EPSRC grants EP/K030930/1, by the EPSRC Centre For Doctoral Training in Industrially Focused Mathematical Modelling (EP/L015803/1) in collaboration with Simula Research Laboratory, and by the Nordic Council of Ministers through Nordforsk grant \#74756
(AUQ-PDE).}}

\author{
  M. Croci${}^\ddagger$\thanks{Mathematical Institute, University of Oxford, Oxford, UK. (\textbf{\url{matteo.croci@maths.ox.ac.uk}}), (\textbf{\url{patrick.farrell@maths.ox.ac.uk}}), (\textbf{\url{mike.giles@maths.ox.ac.uk}}).}
  \and
  M.~B.~Giles\footnotemark[2]
  \and
  M.~E.~Rognes\thanks{Department for Numerical Analysis and Scientific Computing, Simula Research Laboratory, P.O.~Box 134, 1325 Lysaker, Norway. (\textbf{\url{meg@simula.no}}).}
  \and
  P.~E.~Farrell\footnotemark[2]
}

\usepackage{amsopn}
\DeclareMathOperator{\diag}{diag}
\DeclareMathOperator{\E}{\mathbb{E}}
\DeclareMathOperator{\V}{\mathbb{V}}

\DeclareMathOperator{\spn}{\text{span}}
\DeclareMathOperator{\W}{\dot{W}}

\usepackage{todonotes}

\definecolor{myblue}{RGB}{135, 206, 250}

\newcommand{\R}{\mathbb{R}}

\ifpdf
\hypersetup{
  pdftitle={\TheTitle},
  pdfauthor={\TheAuthors}
}
\fi

\begin{document}

\maketitle

\begin{abstract}
  \textbf{Abstract:} When solving stochastic partial differential equations (SPDEs)
  driven by additive spatial white noise, the efficient sampling of
  white noise realizations can be challenging. Here, we present a new
  sampling technique that can be used to efficiently compute white
  noise samples in a finite element method and multilevel Monte Carlo
  (MLMC) setting. The key idea is to exploit the finite element matrix
  assembly procedure and factorize each local mass matrix
  independently, hence avoiding the factorization of a large matrix.
  Moreover, in a MLMC framework, the white noise samples must be
  coupled between subsequent levels. We show how our technique can be
  used to enforce this coupling even in the case of non-nested mesh
  hierarchies. We demonstrate the efficacy of our method with
  numerical experiments. We observe optimal convergence rates for the
  finite element solution of the elliptic SPDEs of interest in 2D and
  3D and we show convergence of the sampled field covariances. In a
  MLMC setting, a good coupling is enforced and the telescoping sum is
  respected.
\end{abstract}


\section{Introduction}

Gaussian fields are ubiquitous in uncertainty quantification to model the uncertainty in spatially dependent parameters. Common applications are in geology, oil reservoir modelling, biology and meteorology \cite{BolinLindgren2011, Harbrecht2016, Lindgren2011, potsepaev2010application}. Here, let $D\subset\R^d$ be an open spatial domain of interest whose closure is a compact subset of $\R^d$. Consider the task of sampling from a zero-mean Gaussian field $u$ of Mat\'ern covariance $\mathcal{C}$:
\begin{align}
\label{eq:Matern}
\mathcal{C}(x,y) = \dfrac{\sigma^2}{2^{\nu-1}\Gamma(\nu)}(\kappa r)^\nu \mathcal{K}_\nu(\kappa r),\ \ r=\Vert x-y\Vert _2,\ \ \kappa = \frac{\sqrt{8\nu}}{\lambda},\ \ x,y\in D,
\end{align}
where $\sigma^2$, $\nu$, $\lambda >0$ are the variance, smoothness parameter and correlation length of the field respectively and $\mathcal{K}_\nu$ is the modified Bessel function of the second kind.

In practice, samples of $u$ are needed only at discrete locations $\bm{x}_1,\dots,\bm{x}_m\in D$ and a simple sampling strategy consists in drawing realizations of a Gaussian vector $\bm{u}\sim\mathcal{N}(0, C)$ with $\bm{u}_i = u(\bm{x}_i)$ and covariance matrix $C_{ij}=\E[u(\bm{x}_i)u(\bm{x}_j)]$. The simplest approach is usually computationally expensive as it requires the factorization of the dense covariance matrix $C\in\R^{m\times m}$. In fact, if we let $\bm{z}\sim\mathcal{N}(0,I)$ be a standard Gaussian vector and we factorize $C = HH^T$ with $H\in\R^{m\times n}$, we can sample $\bm{u}$ as $\bm{u} = H\bm{z}$ since,
\begin{align}
\label{eq:trick}
\E[\bm{u}\bm{u}^T]= \E[H\bm{z}(H\bm{z})^T]=H\E[\bm{z}\bm{z}^T]H^T= HIH^T=C.
\end{align}
A basic form of this approach uses the Cholesky factorization of $C$. In this case $H$ is dense and the factorization has a computational complexity of $O(m^3)$. If the field is smooth enough so that the eigenvalues of $C$ are rapidly decaying this method can be made competitive by using a low-rank approximation instead \cite{Harbrecht2012}. Usually $n$ is taken to be equal to $m$ so that $H$ is square. However this is not necessary for \eqref{eq:trick} to hold. For instance, the sampling strategy we present in this work uses $n>m$. 

Another family of sampling approaches is based on the expansion of the field $u$ as a (possibly finite or truncated) series of basis functions. Different choices of bases yield different methods. Common choices are the basis of the eigenfunctions of $\mathcal{C}(x,y)$ (Karhunen-Lo\`eve expansion), the Fourier basis (circulant embedding \cite{dietrich1997fast}), or a finite element basis (\cite{Bolin2017,Drzisga2017,Lindgren2011,Zhang2016}). The former method is the most flexible as it can be used to sample Gaussian vectors with arbitrary covariance structure. However, it requires either the solution of a dense eigenvalue problem or the factorization of a dense covariance matrix \cite{Harbrecht2012}. If only the largest eigenvalues or a low-rank factorization are needed, this approach is reasonably efficient \cite{Harbrecht2012}. However, if the eigenvalues of $\mathcal{C}(x,y)$ decay slowly ($\nu$ small), such operations become expensive as more terms are needed in the expansion. Circulant embeddings are exact and more efficient, but rely on the use of the fast Fourier transform, the computation of which typically requires simple geometries and uniform structured meshes.

In this paper we consider the finite element basis method. Whittle showed in \cite{Whittle1954} that a Mat\'ern field with covariance given by \eqref{eq:Matern} is the statistically stationary solution that satisfies the linear elliptic PDE,
\begin{align}
\label{eq:white_noise_eqn}
\left(\mathcal{I} - \kappa^{-2}\Delta\right)^{k}u(x,\omega) = \eta \W(\cdot, \omega),\quad x\in \R^d,\quad \omega\in\Omega,\quad \nu = 2k - d/2 > 0,
\end{align}
where $\W$ is spatial Gaussian white noise in $\R^d$, $k>d/4$ and $\Omega$ is a suitable sample space. The notation $\W(\cdot, \omega)$ indicates that $\W$ is almost surely a generalised function (on $\mathbb{R}^d$). Here $\eta$ is a scaling factor that depends on $\sigma$, $\lambda$ and $\nu$, $d\leq 3$ and the equality has to hold almost surely and be interpreted in the sense of distributions. Boundary conditions are not needed as the stationarity requirement is enough for well-posedness \cite{Lindgren2011}. Equation \eqref{eq:white_noise_eqn} has to be solved on the whole $\R^d$. However this is generally not feasible and $\R^d$ is in practice truncated to a bounded domain $D$. In this case, artificial boundary conditions must be prescribed on $\partial D$. Homogeneous Dirichlet or Neumann boundary conditions are often chosen \cite{Bolin2017,Lindgren2011}, although it usually does not matter for practical purposes as the error in the covariance of the field decays rapidly away from the boundary \cite{potsepaev2010application}. After the meshing of $D$, \eqref{eq:white_noise_eqn} can be solved in linear time with the finite element method (FEM) and an optimally preconditioned Krylov solver.
This approach thus scales well in terms of problem size and parallel computation \cite{Drzisga2017}. Moreover, the approach is especially convenient if $u$ appears as a coefficient in a PDE which is solved using the FEM as it might be possible to reuse finite element bases and computations for both equations.

The main focus of this paper is the generation of white noise samples $\W(\cdot, \omega)$ for a given sample point $\omega\in\Omega$. More precisely, we study the efficient sampling of the action $\langle \W, v_h \rangle (\omega)$ of white noise onto a FEM test function $v_h$. In this work we specifically consider equation \eqref{eq:white_noise_eqn} for the sake of simplicity. However, the sampling techniques we describe apply to a wider range of SPDEs with additive spatial white noise forcing (e.g.~see \cite{Du2002,Zhang2016}). While solving such equations is relatively straightforward, the sampling of white noise realizations is not as it requires the sampling of a Gaussian vector with a finite element mass matrix $M$ as covariance. If the finite element spaces involved are other than piecewise constants, $M$ will be sparse, but not diagonal. Hence, its Cholesky factor is usually dense and the sampling requires an offline computational and memory storage cost of $O(m^3)$ for the factorization and an online cost of $O(m^2)$ for each sample.

To resolve this challenge, different approaches have been adopted in the literature. Generally, the idea has been to use a diagonal mass matrix instead; i.e.~an approximate representation using piecewise constants or mass-lumping. Osborn et al.~\cite{Osborn2017} use a two-field reformulation of \eqref{eq:white_noise_eqn} for $k=1$ with Raviart-Thomas elements combined with piecewise constants, while Lindgren et al.~\cite{Lindgren2011} use continuous Lagrange elements and mass lumping. Both methods compute (or approximate) the action of white noise on the FEM test functions. Another option, adopted in \cite{Drzisga2017,Du2002,potsepaev2010application}, is to approximate the white noise itself by a piecewise constant random function that converges in an appropriate weak sense to the exact white noise.

The sampling becomes more complicated when the Mat\'ern field $u$ is needed within a multilevel Monte Carlo (MLMC) framework \cite{giles2008,giles2015multilevel} which requires the coupling of the field between different approximation levels (i.e.~the same sample point $\omega$ must be used on both levels). In turn, this requires the white noise samples on each level to be coupled. Drzisga et al.~\cite{Drzisga2017} enforce this coupling in the nested grid case with the use of a piecewise constant approximation of white noise \cite{Scheichl2017}. Osborn et al.~\cite{Osborn2017} present a technique that enforces the coupling between nested meshes by using techniques from element-based algebraic multigrid (AMG). Their approach does not require a user-provided hierarchy of nested grids as the hierarchy is constructed algebraically. This operation aggregates the elements of a single user-provided grid into clusters which then constitute the elements of the coarse meshes. The resulting aggregated meshes are non-simplicial. Furthermore, Osborn et al.~\cite{Osborn2017scalable} use a hierarchy of nested structured grids on which they enforce the white noise coupling and solve the SPDE \eqref{eq:white_noise_eqn}. The techniques used for the coupling are the same as presented in \cite{Osborn2017}. The sampled Mat\'ern fields are then transferred to a non-nested agglomerated mesh of the domain of interest via a Galerkin projection.

The main contributions of this paper are the following. First, we present a sampling technique for white noise that is exact and that is applicable for a wide range of finite element families, including all types of Lagrange elements. Our technique does not require the expensive factorization of a global mass matrix or a costly two-field splitting of the Laplacian and has linear complexity in the number of degrees of freedom. Second, we introduce a coupling technique for coupling white noise between nested or non-nested meshes, applicable for the same class of finite element families. If non-nested meshes are used, this coupling technique requires the use of a supermesh construction \cite{PatrickPHD,Farrell2011Supermesh,Farrell2009Supermesh}. Third, the existing literature generally focuses on white noise coupling in the $h$-refinement case, i.e.\ when the MLMC hierarchy is defined by meshes of decreasing element size \cite{CharrierMLMC2013,Cliffe2011}. In this paper we also consider the case in which the MLMC levels are defined by increasing the polynomial degree of the FEM interpolant ($p$-refinement).

Although Osborn et al.~\cite{Osborn2017} also work with non-nested meshes, our approach differs significantly from theirs. Osborn et al.~start from one single mesh and algebraically coarsen it to obtain the grid hierarchy. The MLMC levels are thus generated algebraically. Our approach operates on a given arbitrary mesh hierarchy and the MLMC levels are defined geometrically. In our case, every mesh in the hierarchy is simplicial, and it is thus possible to use standard FEM error estimates (if available) to estimate \textit{a priori} the MLMC convergence parameters \cite{CharrierMLMC2013,TeckentrupMLMC2013}.

We adopt the same embedded domain strategy as Osborn et al.~\cite{Osborn2017}. The advantage of this strategy is that the sampled Mat\'ern field can be transferred to the computational domain of interest exactly and at negligible cost. However, in practical applications defined over complex geometries, a sequence of nested meshes might not be available, making the white noise coupling challenging. This motivated us to design an algorithm that can be used to enforce the coupling between non-nested meshes as well.

The paper is structured as follows. In Section \ref{sec:notation} we introduce the notation we use and we give a brief description of the MLMC method. In Section \ref{sec:sec2} we present an overview of the white noise SPDE sampling approach for Mat\'ern fields and we suggest a simple FEM scheme for the solution of \eqref{eq:white_noise_eqn}. Moreover, we describe the white noise sampling problem for the cases where both independent and coupled realizations are needed. In Section \ref{sec:white_noise_sampling} we describe our new sampling technique that allows the sampling of independent and coupled white noise realizations efficiently. In Section \ref{sec:num_results} we present numerical results corroborating the theoretical results and demonstrating the performance of the technique. Finally we summarize the results of the paper in Section \ref{sec:conclusions}.

\section{Notation and preliminaries}
\label{sec:notation}
\subsection{Notation}
In this paper we adopt the following notation.

$\bm{L^2}$ \textbf{inner product.} For an open domain $D \subseteq \R^d$, we let $(\cdot, \cdot)$
denote the $L^2(D)$ inner product where $L^2(D)$ is the standard
Hilbert space of square-integrable functions on $D$.

\textbf{Real-valued random variables.} For a given $\sigma$-algebra $\mathcal{A}$ and probability measure $\mathbb{P}$ let $(\Omega,\mathcal{A},\mathbb{P})$ be a probability space and let $L^2(\Omega, \R)$ indicate the space of real-valued random variables of finite second moment.

\textbf{Generalized stochastic fields.} Following the definition introduced by It\^{o} \cite{Ito1954} we denote with ${\mathscr{L}(L^2(D), L^2(\Omega, \R))}$ the space of generalized stochastic fields that are continuous linear mappings from $L^2(D)$ to $L^2(\Omega,\R)$. For a given $\xi\in \mathscr{L}(L^2(D), L^2(\Omega, \R))$ we indicate the action of $\xi$ onto a function $\phi\in L^2(D)$ with the notation $\xi(\phi) = \langle \xi, \phi \rangle$.

\textbf{Subsets of compact closure.} Given an open domain $G\subseteq D$, we write $G\subset\subset D$ to indicate that the closure of $G$ is a compact subset of $D$.

\textbf{Nested and non-nested meshes.} Let $T_a$ and $T_b$ be two tessellations of $D$. We say that $T_a$ is nested within $T_b$ if $\textnormal{vertices}(T_a)\subseteq \textnormal{vertices}(T_b)$ and if for each element $e\in T_a$ there exists a set of elements $E\subseteq T_b$ such that $e=\bigcup_{\hat{e}_i\in E}\hat{e}_i$. We say that $T_a$ and $T_b$ are non-nested if $T_a$ is not nested within $T_b$ and vice-versa.

Additionally, we will use the following definition of \textbf{white noise}.
\begin{definition}[White noise, see example 1.2 and lemma 1.10 in \cite{Hida1993}]
	\label{def:white_noise}
	Let $D\subseteq\R^d$ be an open domain. The white noise $\W\in \mathscr{L}(L^2(D), L^2(\Omega, \R))$ is a generalized stochastic field such that for any collection of $L^2(D)$ functions $\{\phi_i\}$, if we let $b_i = \langle \W, \phi_i \rangle$, then $\{b_i\}$ are joint Gaussian random variables with zero mean and covariance given by $\mathbb{E}[b_ib_j]=(\phi_i,\phi_j)$.
\end{definition}

\subsection{The multilevel Monte Carlo method}
Let $u(x,\omega)$ for $x\in \R^d$, $\omega\in \Omega$ be the solution of an SPDE of interest, e.g.~\eqref{eq:white_noise_eqn}. Generally we are interested in computing an output functional $P$ of $u$, namely
\begin{align}
\label{eq:output_functional}
P(\omega) = \mathcal{P}[x,u(x,\omega)](\omega).
\end{align}
Here we assume that $P(\omega)$ is scalar-valued with
bounded second moment, i.e.\ $P\in L^2(\Omega, \R)$.

For instance, $\mathcal{P}$ could be the average or the $L^2$ norm of $u$ over its domain of definition. A more complicated, yet common, case is when the computation of $\mathcal{P}$ requires the solution of an additional equation. A typical example is when $u$ is a Mat\'ern field satisfying \eqref{eq:white_noise_eqn} and $u$ appears in the permeability coefficient of another elliptic PDE \cite{Cliffe2011}. In this case $\mathcal{P}$ is usually a functional of the solution of the latter equation \cite{Cliffe2011}. Although the techniques we describe apply to a wider range of problems, this is the framework considered in this paper.

To quantify the propagation of uncertainty from the input $u$ to the output functional
of interest $P$, one may estimate the expected value $\E$ and variance
$\V$ of $P(\omega)$. When $P$ can be approximated at different levels $\ell=1,\dots,L$ of increasing accuracy and cost, $\E[P]$ and $\V[P]$ can be estimated efficiently with the multilevel Monte Carlo (MLMC) method \cite{giles2008}. Letting $P_\ell$ be the approximated value at level $\ell$, we can approximate $\E[P]$ as
\begin{align}
\label{eq:telescoping_sum}
\E[P] \approx \E[P_L] = \sum\limits_{\ell = 1}^L \E[P_\ell - P_{\ell-1}],\quad P_0 \equiv 0.
\end{align}
The telescoping sum on the right hand side is at the heart of the MLMC strategy: by approximating each term in the sum with standard MC we obtain the MLMC estimator,
\begin{align}
\label{eq:MLMC_estimator}
\E[P] \approx \bar{P} = 
\sum\limits_{\ell = 1}^L\left[\frac{1}{N_\ell}\sum\limits_{n=1}^{N_\ell} \left (P_\ell(\omega^n_\ell) - P_{\ell  - 1}(\omega^n_\ell) \right )\right],
\end{align}
where $\omega^n_\ell\in\Omega$ is the $n$-th sample point on level $\ell$.

In the context of approximating $u$ and \emph{a fortiori} $P$ using a finite element method, the levels of accuracy can be defined by using a hierarchy of meshes ($h$-refinement) or by increasing the polynomial degree of the finite elements used ($p$-refinement). As the variance is yet another expectation, the same strategy as described here for $\E[P]$ applies to $\V[P]$. 

The increased efficiency of MLMC with respect to standard Monte Carlo relies on the assumption that on coarse levels (small $\ell$), many samples are needed for an accurate estimate of the expected value, but each sample is inexpensive to compute. On fine levels (large $\ell$) sampling is expensive, but the variance is small due to the fact that the levels are coupled, i.e.~the sample point $\omega^n_{\ell}$ is the same for both $P_\ell(\omega^n_\ell)$ and $P_{\ell  - 1}(\omega^n_\ell)$. The coupling makes $P_\ell(\omega^n_\ell)$ and $P_{\ell  - 1}(\omega^n_\ell)$ strongly correlated. This aspect diminishes the variance of their difference, and therefore fewer samples are required to estimate the expected value. An alternative interpretation is that once $\omega^n_{\ell}$ is fixed, the two terms $P_\ell(\omega^n_\ell)$ and $P_{\ell  - 1}(\omega^n_\ell)$ are two approximations of different accuracy of the same deterministic problem and hence their difference becomes smaller as the discretization approaches the infinite dimensional solution.

MLMC can be seen as a variance reduction technique in which the coupling between the levels is one of the key elements. If the coupling is not enforced correctly so that the samples of $P_\ell$ and $P_{\ell-1}$ become independent, then the variance of each term of the telescoping sum in \eqref{eq:MLMC_estimator} increases, significantly harming its efficiency and convergence properties.

The convergence and cost of the MLMC estimator \eqref{eq:MLMC_estimator} is given by the following theorem:
\begin{theorem}[MLMC convergence \cite{Cliffe2011, giles2008}]
	\label{th:MLMC_convergence}
	Let $C_\ell$ be the cost of computing one sample of $P_\ell$ on level $\ell$. Suppose that there are positive constants $\alpha$, $\beta$, $\gamma$ such that $\alpha\geq\min(\beta,\gamma)$ and
	\begin{enumerate}[label={\arabic*)}, leftmargin=1cm]
		\item $|\E[P_\ell - P]|\hspace{6pt} \lesssim 2^{-\alpha\ell}$,
		\item $\V[P_\ell - P_{\ell-1}] \lesssim  2^{-\beta\ell}$,
		\item $C_\ell \hspace{45pt} \lesssim 2^{\gamma\ell}$.
	\end{enumerate}
	Then, for any $\varepsilon < e^{-1}$, there exists a value $L$ and a sequence $\{N_\ell\}_{\ell=1}^L$ such that,
	\begin{align}
	\label{eq:MSE}
	\textnormal{error}(\bar{P}) \equiv \E[(\bar{P} - \E[P])^2]^{1/2} \leq \varepsilon,
	\end{align}
	and the total cost $C_{\text{tot}}$ satisfies
	\begin{align}
	C_{\text{tot}} := \sum\limits_{\ell=1}^L N_\ell C_\ell \lesssim \left\{\begin{array}{lr} \varepsilon^{-2}, & \beta > \gamma,\\\varepsilon^{-2}(\log\varepsilon)^2, & \beta = \gamma,\\\varepsilon^{-2 - (\gamma - \beta)/\alpha}, & \beta < \gamma.\end{array}\right.
	\end{align}
\end{theorem}
The values of $L$ and $\{N_\ell\}_{\ell=1}^L$ are estimated automatically in standard MLMC algorithms, for further details see \cite{giles2008,giles2015multilevel}. The values of the MLMC parameters $\alpha$, $\beta$, $\gamma$ are sometimes known \textit{a priori} \cite{CharrierMLMC2013,TeckentrupMLMC2013}, otherwise they need to be estimated. In the uniform $h$-refinement case, we have $\gamma=d$ and $h_\ell \sim 2^{-c\ell}$, where $h_\ell$ is the level $\ell$ mesh size and $c>0$ \cite{Cliffe2011}.

\section{The finite element approach to Mat\'ern field sampling}
\label{sec:sec2}

$\,$
In this section we describe the practical aspects of the numerical solution of \eqref{eq:white_noise_eqn} when either independent (standard Monte Carlo) or coupled (MLMC) Mat\'ern field samples are needed. As we will see, the main complication lies in the sampling of white noise realizations.

Note that the results we present on white noise sampling can be applied to a wider class of elliptic or parabolic SPDEs with additive spatial white noise (e.g.~see \cite{Du2002,Zhang2016}).

\subsection{Finite element solution of elliptic PDEs with white noise forcing}
\label{sec:singleFEMsol}

The solutions of the linear elliptic PDE \eqref{eq:white_noise_eqn}
correspond to a Mat\'ern field with covariance given by
\eqref{eq:Matern}. In this paper we assume $k$ to be a positive
integer, although it is possible to work with non-integer values as
well \cite{BolinLindgren2011}. The scaling factor $\eta$ in
\eqref{eq:white_noise_eqn} is given by
\begin{align}
\eta = \frac{\sigma}{\hat{\sigma}},\quad\text{where}\quad\hat{\sigma}^2 = \dfrac{\Gamma(\nu)\ \nu^{d/2}}{\Gamma(\nu+d/2)}\left(\frac{2}{\pi}\right)^{d/2}\lambda^{-d},
\end{align}
where $\Gamma(x)$ is the Euler gamma function \cite{Lindgren2011}. Note that if $d=2$ then $\hat{\sigma}^2 = (2/\pi)\lambda^{-2}$, and for $\nu\rightarrow\infty$, $\hat{\sigma}^2 = (2/\pi)^{d/2}\lambda^{-d}$.

Solving \eqref{eq:white_noise_eqn} over the whole of $\R^d$ is generally not feasible. Instead, $\R^d$ is typically truncated to a bounded domain $D\subset\joinrel\subset\R^d$ and some boundary conditions are chosen, usually homogeneous Neumann or Dirichlet \cite{Bolin2017,Lindgren2011}. In what follows, we assume that the Mat\'ern field sample is needed on a domain $G\subset\joinrel\subset D$. If $D$ is sufficiently large in the sense that the distance between $\partial D$ and $\partial G$ is larger than the correlation length $\lambda$ then the error introduced by truncating $\R^d$ to $D$ is negligible \cite{Drzisga2017,potsepaev2010application}.

After truncating the domain, \eqref{eq:white_noise_eqn} can be rewritten in the following iterative form,
\begin{align}
\label{eq:iterative_white_noise_eqn}
\begin{dcases}
\begin{array}{lclcl}
u_1 - \kappa^{-2}\Delta u_1 = \eta \W && \text{in }  D &&\\
u_{j+1} - \kappa^{-2}\Delta u_{j+1} = u_j && \text{in }  D,&& j=1,\dots,k-1,\\
u_{j+1}=0 && \text{on } \partial D, && j=0,\dots,k-1,
\end{array}
\end{dcases}
\end{align}
where $u\equiv u_k$. This is the approach suggested by Lindgren et al.~in \cite{Lindgren2011}. As the main focus of this paper is on white noise sampling, we will restrict our attention to the $k=1$ case and we will set $\eta=1$ from now on, in which case \eqref{eq:iterative_white_noise_eqn} reduces to
\begin{equation}
\label{eq:truncated_white_noise_eqn}
\begin{split}
u - \kappa^{-2}\Delta u &= \W \quad \text{in } D, \\
u &= 0 \quad \text{on } \partial D.
\end{split}
\end{equation}
Existence and uniqueness of solutions to~\eqref{eq:truncated_white_noise_eqn} was proven in \cite{Benfatto1980} and in \cite{Buckdahn1989}.

We will solve \eqref{eq:truncated_white_noise_eqn} using the finite element method. Let $V_h=\spn(\phi_1,\dots,\phi_m)\subseteq H^1_0(D)$ be a suitable finite element approximation subspace (e.g.\ the $\phi_i$ could be continuous Lagrange basis functions defined relative to a triangulation $D_h$ of $D$). A discrete weak form of \eqref{eq:truncated_white_noise_eqn} then reads: find $u_h\in V_h$ such that
\begin{align}
(u_h, v_h) + \kappa^{-2}(\nabla u_h,\nabla v_h) = \langle \W, v_h \rangle
\quad \text{for all } v_h\in V_h.
\end{align}

The coefficients of the basis function expansion for $u_h$, i.e.~the $u_i$ such that ${u_h=\sum_{i=1}^mu_i\phi_i}$, are given by the solution of a linear system
\begin{align}
\label{eq:linear_system}
A\bm{u} = \bm{b},\quad\text{with}\quad A_{ij}=(\phi_i,\phi_j) + \kappa^{-2}(\nabla\phi_i,\nabla\phi_j),\quad b_i = \langle \W, \phi_i \rangle.
\end{align}
This linear system \eqref{eq:linear_system} can be solved in $O(m)$ time by using an optimally preconditioned Krylov solver such as the conjugate gradient method preconditioned with geometric or algebraic multigrid. We remark that since the elliptic operator is the same in all equations of \eqref{eq:iterative_white_noise_eqn}, the same finite element basis and solver can be reused to compute all the $u_j$ for the case where $k > 1$. 

By Definition \ref{def:white_noise}, $\bm{b}$ satisfies
\begin{align}
\bm{b}\sim\mathcal{N}(0,M),\quad M_{ij} = (\phi_i,\phi_j),
\end{align}
i.e.~$\bm{b}$ is a zero-mean Gaussian vector with the finite element mass matrix $M$ as covariance matrix. Sampling white noise realizations can thus be accomplished by sampling a Gaussian vector of mass matrix covariance.

In Section \ref{sec:white_noise_sampling}, we present a factorization
of $M$ in the form $H H^T$ (cf.~\eqref{eq:trick}) that is both sparse
and computationally efficient to compute, thus allowing for efficient
sampling of white noise.

\subsection{Multilevel white noise sampling/white noise coupling condition} We now consider the case in which coupled Mat\'ern field realizations are needed in a MLMC setting, i.e.~we want to draw samples of $u_{\ell}(x,\omega)$ and $u_{\ell-1}(x,\omega)$ at two different levels of accuracy $\ell$ and $\ell-1$ for the same $\omega\in \Omega$. Since the only stochastic element present in \eqref{eq:truncated_white_noise_eqn} is white noise, it is sufficient to use the same white noise sample on both levels to enforce the coupling requirement.

More precisely, let $V^\ell$ and $V^{\ell-1}$ be the finite element spaces on level $\ell$ and $\ell-1$ respectively for $\ell > 1$. We consider the following two variational problems coupled by a common white noise sample: find $u_\ell\in V^\ell=\spn(\phi_1^\ell,\dots,\phi^\ell_{m_\ell})$ and $u_{\ell-1}\in V^{\ell-1}=\spn(\phi_1^{\ell-1},\dots,\phi^{\ell-1}_{m_{\ell-1}})$ such that for $\omega^n_\ell\in \Omega$
\begin{align}
\label{eq:coupled1}
(u_{\ell},v_{\ell}) + \kappa^{-2}(\nabla u_{\ell},\nabla v_{\ell})         &= \langle \W,v_{\ell} \rangle (\omega^n_\ell),\quad\hspace{10.5pt}\text{for all } v_{\ell}\in V^\ell,\\
(u_{\ell-1},v_{\ell-1}) + \kappa^{-2}(\nabla u_{\ell-1},\nabla v_{\ell-1}) &= \langle \W,v_{\ell-1} \rangle (\omega^n_\ell),\quad\text{for all } v_{\ell-1}\in V^{\ell-1}.
\label{eq:coupled2}
\end{align}
where the terms on the right hand side are coupled in the sense that they are centered Gaussian random variables with covariance
$\E[\langle \W,v_{l} \rangle \langle \W,v_{s} \rangle] = (v_l,v_s)$ for $l,s\in\{\ell,\ell-1\}$, as given by definition \ref{def:white_noise}.

Let $\bm{u}_\ell\in\R^{m_\ell}$ and $\bm{u}_{\ell-1}\in\R^{m_{\ell-1}}$ be the vectors of the finite element expansion coefficients of $u_\ell$ and $u_{\ell-1}$, respectively. Following the same approach as in Section \ref{sec:singleFEMsol}, we note that the coefficient vectors solve the following block diagonal linear system,
\begin{align}
\left[
\begin{array}{c|c}%
A^\ell & 0 \\\hline\\[-1em]
0 & A^{\ell-1}
\end{array}
\right]
\left[
\begin{array}{l}%
\bm{u}_\ell \\
\bm{u}_{\ell-1}
\end{array}
\right] = 
\left[
\begin{array}{l}%
\bm{b}_\ell \\
\bm{b}_{\ell-1}
\end{array}
\right].
\end{align}
Alternatively, by letting $\bm{u} = [\bm{u}_\ell,\ \bm{u}_{\ell-1}]^T$, $\bm{b} = [\bm{b}_\ell,\ \bm{b}_{\ell-1}]^T$ and $A = \diag(A^\ell,A^{\ell-1})$, we can write this as
\begin{align}
A\bm{u} = \bm{b}. 
\end{align}
This system can be solved in linear time with an optimal solver \cite{elman2014finite}.

Furthermore, by definition \ref{def:white_noise},
\begin{align}
\label{eq:coupling_condition1}
\bm{b} \sim\mathcal{N}(0,M),
\end{align}
where $M$ can be expressed in block structure as
\begin{align}
\label{eq:coupling_condition2}
M = \left[
\begin{array}{c|c}%
M^\ell & M^{\ell,\ell-1} \\\hline\\[-1em]
(M^{\ell,\ell-1})^T & M^{\ell-1}
\end{array}
\right],
\quad\text{with}\quad M^{\ell,k}_{ij} = (\phi_i^\ell,\phi_j^{k})
\text{ and } M^{k}_{ij} = (\phi_i^k,\phi_j^{k}).
\end{align}
If we were using independent white noise samples for \eqref{eq:coupled1} and \eqref{eq:coupled2}, then the off-diagonal blocks of $M$ would vanish. Conversely, the presence of the mixed mass matrix $M^{\ell,\ell-1}$ stems from the use of the same white noise sample on both levels. For this reason, we will refer to equations \eqref{eq:coupling_condition1} and \eqref{eq:coupling_condition2} as the \emph{coupling condition}.

Thus, the problem of sampling coupled Mat\'ern fields in the context of MLMC again reduces to the sampling of a Gaussian vector with a mass matrix as covariance. However, two additional complications arise. First, $M$ is potentially much larger and not necessarily of full rank (consider the case in which $V^\ell = V^{\ell-1}$, then $M^\ell=M^{\ell-1}=M^{\ell,\ell-1}$). Second, to assemble $M^{\ell,\ell-1}$ we need to compute integrals involving basis functions possibly defined over different, non-nested meshes, which is non-trivial. In Section \ref{sec:white_noise_sampling}, we present a sampling technique that addresses both issues. A supermesh construction \cite{Farrell2011Supermesh, Farrell2009Supermesh} is required in the non-nested mesh case.

\subsection{Embedded meshes and non-nested grids}

We adopt the same embedded mesh strategy as presented by Osborn et
al.~\cite{Osborn2017}. We assume that the Mat\'ern field sample is
needed on a user-provided mesh $G_h$ of the domain $G$ and we take $D$
to be a larger $d$-dimensional box such that the distance between
$\partial G$ and $\partial D$ is at least $\lambda$. With modern
meshing software, such as Gmsh \cite{gmsh}, it is possible to then triangulate $D$ and obtain a
mesh $D_h$ in such a way that $G_h$ is nested within $D_h$, i.e., each
element and vertex of $G_h$ is also an element or vertex of $D_h$. We
then refer to $G_h$ as embedded in $D_h$ or to $G_h$ as an embedded
mesh (in $D_h$). The main advantage of an embedded $G_h$ in $D_h$ is
that once \eqref{eq:truncated_white_noise_eqn} is solved on $D_h$ the
sampled Mat\'ern field $u$ can be exactly transferred onto $G_h$ at
negligible cost. Conversely, if $G_h$ is not embedded in $D_h$, an
additional interpolation step would be required, thus increasing the
cost of each sample.

In the MLMC framework with $h$-refinement, we assume that we are given a possibly non-nested user-provided mesh hierarchy $\{G^\ell_h\}_{\ell=1}^L$. We accordingly generate a hierarchy of meshes $\{D^\ell_h\}_{\ell=1}^L$ on which to perform the sampling. If the $\{D^\ell_h\}_{\ell=1}^L$ are nested, then the techniques used in \cite{Drzisga2017} and \cite{Osborn2017} can be used to couple the white noise between MLMC levels. However, in the case in which the user-provided meshes $\{G^\ell_h\}_{\ell=1}^L$ are non-nested, these methods are not compatible with the embedded mesh strategy.

Clearly, non-nested grid hierarchies appear naturally in practical computations on complex geometries. For instance, grid hierarchies generated from CAD geometries or through coarsening of a single fine mesh are generally non-nested. Thus tackling couplings across non-nested meshes is crucial for non-trivial applications. As we will see in the next section, such couplings can be achieved at a small offline cost.

\section{White noise sampling}
\label{sec:white_noise_sampling}
In this section we introduce a new technique for sampling white noise efficiently. We first address the basic case in which independent white noise samples are needed before considering the more complicated case in which coupled samples are required.

\subsection{Sampling of independent white noise realizations}
\label{sec:independent_white_noise_samples}
As discussed in the previous section, the sampling of independent white noise realizations defined over a meshed domain can be cast as the sampling of a Gaussian vector $\bm{b}$ of covariance matrix given by a finite element mass matrix $M \in \R^{m \times m}$. In turn, efficient sampling of such a Gaussian vector typically involves computing a factorization of $M=HH^T$. If a Cholesky factorization is used, such sampling may become costly, with a $O(m^3)$ factorization cost and $O(m^2)$ cost per sample. In what follows, we present an alternative factorization strategy which has $O(m)$ fixed cost and $O(m)$ cost per sample.

The core idea is to work element-wise instead of factorizing a global mass matrix. To illustrate, consider a standard finite element assembly of $M$ over a mesh with $n$ elements (i.e.~cells) and $m_e$ degrees of freedom on each element. Local mass matrices $M_e$ of size $m_e\times m_e$ are computed on each mesh element $e$ before aggregation to form the global mass matrix $M$. The overall assembly operation can be written in matrix form,
\begin{align}
M = L^T \diag_e(M_e) L,
\end{align}
(see e.g.~\cite{Wathen1987}), where $\diag_e(M_e)$ is a block diagonal matrix of size $nm_e\times nm_e$ with the local mass matrices on the diagonal and $L$ is a Boolean assembling matrix of size $nm_e\times m$ such that $L^T = [L_1^T\ \dots\ L_{n}^T]$ and the $L_e$ are Boolean matrices of size $m_e\times m$ that encode the local-to-global map. Note that each row of $L$ has exactly one non-zero entry \cite{Wathen1987}.

We can now factorize each local mass matrix $M_e$ independently with a standard Cholesky factorization to obtain $M_e=H_eH_e^T$ for each element $e$. We then have
\begin{align}
M = L^T\diag_e(H_eH_e^T)L = (L^T\diag_e(H_e)) (L^T\diag_e(H_e))^T = HH^T,
\end{align}
with $H \equiv L^T\diag_e(H_e)$, and we can sample $\bm{b}$ by computing
\begin{align}
\bm{b} = Hz,\quad\text{with}\quad z\sim\mathcal{N}(0,I),\quad z\in\R^{m_en},
\end{align}
since, cf.~\eqref{eq:trick},
\begin{align}
\E[\bm{b}\bm{b}^T] = H\E[\bm{z}\bm{z}^T]H^T= (L^T\diag_e(H_e))I(L^T\diag_e(H_e))^T=M.
\end{align}
\begin{remark}
	\label{obs:global_to_local}
	This sampling strategy allows the splitting of a large global sampling problem into separate small local sampling problems. In fact, if for each element $e$ we let $\bm{z}_e\sim\mathcal{N}(0,I)$ be a small standard Gaussian vector of length $m_e$, we can rewrite $\bm{b} = H\bm{z}$ as
	\begin{align}
	\label{eq:trick_global_to_local}
	\bm{b} = H\bm{z} = \sum\limits_{e=1}^nL^T_eH_e\bm{z}_e = \sum\limits_{e=1}^nL^T_e\bm{b}_e,
	\end{align}
	where $\bm{b}_e\sim\mathcal{N}(0,M_e)$ is sampled locally. The problem of sampling a global mass matrix covariance Gaussian vector then eventually reduces to the sampling of $n$ independent local mass matrix covariance Gaussian vectors. This sampling approach is therefore trivially parallelizable. 
\end{remark}

Note that this sampling strategy is efficient since the local Cholesky factorizations can be computed in $O(m_e^3n)$ time and the $L_e$ factors can be applied matrix-free for a total $O(m_e^3n)$ factorization cost and an $O(m_e^2n)$ memory and sampling cost. 

\begin{remark}
	\label{remark:Lagrange}
	In the case in which the transformation to the reference element is affine (such as with Lagrange elements on simplices) this operation can be made much more efficient by noting that the local mass matrices on each element are always the same up to a multiplicative factor, namely $M_e/|e| =$ const for all $e$, where $|e|$ is the measure of the element. It is therefore sufficient to factorise a single local mass matrix and to store its Cholesky factor, yielding a negligible $O(m_e^3)$ and $O(m_e^2)$ factorization and memory cost respectively.
\end{remark}

We note that the standard Gaussian vector $\bm{z}$ used to compute $\bm{b}$ is of size $m_en$ which is larger than if a Cholesky factorization was used\footnote{A Cholesky factor would be of size $m \times m$, yielding a standard Gaussian vector of length $m$.}. In fact, unlike the Cholesky factor, the matrix $H$ here is not square. However, in comparison to the cost of solving \eqref{eq:truncated_white_noise_eqn}, the sampling cost of the extra Gaussian variables is negligible.

\subsection{Sampling coupled white noise realizations for MLMC}

We now turn to consider the case of sampling coupled white noise. In
what follows we consider the general setting in which the MLMC levels
are defined using $h$-refinement and the mesh hierarchy is
non-nested. At the end of the section, we provide some remarks on the
simpler cases in which the function spaces that define the hierarchy
are nested (e.g.~the grids are nested or $p$-refinement is used).

\subsubsection{Supermesh construction and global mass matrix assembly}

Consider the case of sampling $\bm{b}\sim\mathcal{N}(0,M)$, where $M\in\R^{m\times m}$ is given by \eqref{eq:coupling_condition2}. The assembly of the off-diagonal blocks of $M$ requires the computation of inner products between basis functions of different FEM approximation subspaces. To address this problem, we use a \emph{supermesh} construction defined as follows.
\begin{definition}[Supermesh, \cite{Farrell2011Supermesh, Farrell2009Supermesh}]
	Let $D\subset\joinrel\subset\R^d$ be an open domain and let $T_a$, $T_b$ be two tessellations of $D$. A supermesh $S$ of $T_a$ and $T_b$ is a common refinement of $T_a$ and $T_b$. More specifically, $S$ is a triangulation of $D$ such that:
	\begin{enumerate}[leftmargin=1cm]
		\item $\textnormal{vertices}(T_a) \cup \textnormal{vertices}(T_b) \subseteq \textnormal{vertices}(S)$,
		\item $\textnormal{volume}(e_S \cap e) \in \{0, \textnormal{volume}(e_S)\}$ for all elements $e_S\in S$, $e\in (T_a \cup T_b)$.
	\end{enumerate}
\end{definition}
The first condition means that every parent mesh vertex must also be a vertex of the supermesh, while the second states that every supermesh element is completely contained within exactly one element of either parent mesh \cite{Farrell2009Supermesh}. As stated in \cite[Lemma 2]{Farrell2009Supermesh}, supermesh elements always lie within the intersection of a single pair of parent mesh elements. The supermesh construction is not unique \cite{Farrell2009Supermesh}. We show an example of supermesh construction in Figure \ref{fig:supermesh}. Efficient algorithms for computing the supermesh are available \cite{libsupermesh-tech-report}. 

\begin{remark}[On the complexity of the supermesh construction]
	If the supermesh construction is performed with a local supermeshing algorithm, then its complexity is $O(n_\ell + K)$, where $n_\ell$ is the number of elements of $D_h^\ell$ and $K$ is the number of intersecting elements \cite{PatrickPHD}. The number of supermesh elements is proportional to the number of intersecting elements and it is therefore $O(K)$. Theoretically, $K$ is bounded by $K \leq c(d) n_\ell n_{\ell-1}$, where $n_{\ell-1}$ is the number of elements of $D_h^{\ell-1}$, $c = 4$ in 2D and $c = 45$ in 3D \cite{PatrickPHD, Farrell2009Supermesh}. In practice, this is a pessimistic bound. If a typical element of the first mesh intersects with $\bar{k}$ elements of the second mesh, then the total number of intersections is $K=O(\bar{k}n_\ell)$. For instance, the libsupermesh software \cite{libsupermesh-tech-report} uses heuristics to eliminate near-degenerate supermesh elements and ensures that $\bar{k}$ is always bounded by a constant. For practical computations, the supermesh construction and the number of supermesh elements is therefore $O(n_\ell)$.
\end{remark}

\begin{figure}[h!]
	\centering
	\includegraphics[trim=0cm 0cm 0cm 0cm, clip=true, width=.8\textwidth]{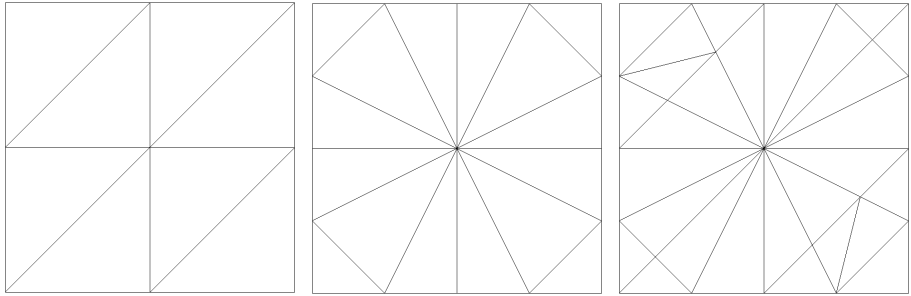}
	\centering
	\caption{An example of a supermesh construction. The first two meshes on the left are the parent meshes and the mesh on the right is a supermesh. As stated in \cite[Lemma 2]{Farrell2009Supermesh}, every supermesh element is completely contained within a unique pair of parent mesh elements.}
	\label{fig:supermesh}
\end{figure}

Evaluating \eqref{eq:coupling_condition2} involves $L^2$-inner products of functions that are only \emph{piecewise} polynomial on each element of $D_h^\ell$ and $D_h^{\ell-1}$. This lack of smoothness affects the convergence of standard quadrature schemes. The supermesh construction provides a resolution to this problem: on each element of a supermesh of $D_h^\ell$ and $D_h^{\ell-1}$ the integrands are polynomial and standard quadrature schemes apply.

Overall, our strategy for evaluating \eqref{eq:coupling_condition2} is to
construct a supermesh of each pair of meshes $D_h^\ell$, $D_h^{\ell-1}$. Note that since each supermesh element lies in the intersection of exactly one pair $(e_\ell, e_{\ell-1})$ of parent mesh elements $e_\ell\in D_h^\ell$, $e_{\ell-1}\in D_h^{\ell-1}$, we only need to account for the basis functions that are non-zero over $e_\ell$ and $e_{\ell-1}$. Let $m_{e_\ell}$ and $m_{e_{\ell-1}}$ denote the number of degrees of freedom defined by the finite element spaces $V^{\ell}$ and $V^{\ell-1}$ over elements $e_\ell$ and $e_{\ell-1}$ respectively. Then, only the inner products between $m_e = m_{e_\ell} + m_{e_{\ell-1}}$ basis functions will be non-zero.

We can thus assemble $M$ given by \eqref{eq:coupling_condition2} by the following two-step algorithm.
\begin{enumerate}[leftmargin=1cm]
	\item Let $n$ be the number of supermesh elements. For each supermesh element $e$, use quadrature rules over $e$ to compute the local mass matrix
	\begin{align}
	\label{eq:local_supermesh_mass_matrix}
	M_e = \left[
	\begin{array}{c|c}%
	M^{\ell}_e & M^{\ell,\ell-1}_e \\\hline\\[-1em]
	(M^{\ell,\ell-1}_e)^T & M^{\ell-1}_e
	\end{array}
	\right],\quad (M^{\ell,\ell-1}_e)_{ij} = \int_e \phi_i^\ell\phi_j^{\ell-1}\text{ dx},
	\end{align}
	where $\{\phi_i^\ell\}_{i=1}^{m_{e_\ell}}$ and $\{\phi_j^{\ell-1}\}_{j=1}^{m_{e_{\ell-1}}}$ are sets of the basis functions of $V^\ell$ and $V^{\ell-1}$ respectively that have non-zero support over $e$. $M_e$ is of size $m_e \times m_e$, $M^{\ell}_e$ is of size $m_{e_\ell}\times m_{e_\ell}$, $M^{\ell-1}_e$ is of size $m_{e_{\ell-1}}\times m_{e_{\ell-1}}$ and $M^{\ell,\ell-1}_e$ is of size $m_{e_\ell}\times m_{e_{\ell-1}}$. 
	\item
	Let $L^\ell$ and $L^{\ell-1}$ be the supermesh assembling
	matrices of the finite element spaces $V^\ell$ and
	$V^{\ell-1}$ respectively, mapping the local supermesh cell
	degrees of freedom to the global degrees of freedom of $V^\ell$. Assemble
	the local supermesh contributions together with
	\begin{align}
	\label{eq:coupling_condition3}
	M =  
	\left[
	\begin{array}{c|c}%
	(L^\ell)^T\diag_e(M^{\ell}_e)L^\ell & (L^\ell)^T\diag_e(M^{\ell,\ell-1}_e)L^{\ell-1} \\\hline\\[-1em]
	(L^{\ell-1})^T\diag_e(M^{\ell,\ell-1}_e)^TL^\ell & (L^{\ell-1})^T\diag_e(M^{\ell-1}_e)L^{\ell-1}
	\end{array}
	\right] .
	\end{align}
\end{enumerate}
Observe that \eqref{eq:coupling_condition3} and \eqref{eq:coupling_condition2} agree since
\begin{align}
\begin{array}{llr}
M^l & = (L^l)^T\diag_e(M^{l}_e)L^l,&\text{for }l\in\{\ell, \ell-1\},\\
M^{\ell,\ell-1} & = (L^\ell)^T\diag_e(M^{\ell,\ell-1}_e)L^{\ell-1}.& 
\end{array}
\end{align}
Note that the above is again just the assembly of the contributions of each supermesh element to the global mass matrices in matrix form. As we will see next, we actually do not need to assemble $M$, but only the local mass matrices $M^\ell_e$ and $M^{\ell,\ell-1}_e$ for each supermesh element $e$.

\subsubsection{From global to local: the local coupling condition}

We again use a divide-and-conquer strategy to split the global sampling problem into smaller local subproblems (cf.~remark \ref{obs:global_to_local}). Suppose that we can sample a local Gaussian vector $\bm{b}_e\sim\mathcal{N}(0,M_e)$ on each supermesh element $e$. We can then separate $\bm{b}_e$ into two Gaussian vectors $\bm{b}^\ell_e$ and $\bm{b}^{\ell-1}_e$ such that $\bm{b}_e = [(\bm{b}^\ell_e)^T, (\bm{b}^{\ell-1}_e)^T]^T$ and
\begin{align}
\label{eq:local_coupling_condition}
\bm{b}^\ell_e\sim\mathcal{N}(0,M^\ell_e),\quad\bm{b}^{\ell-1}_e\sim\mathcal{N}(0,M^{\ell-1}_e),\quad\E[\bm{b}^\ell_e(\bm{b}^{\ell-1}_e)^T] = M^{\ell,\ell-1}_e.
\end{align}
Since \eqref{eq:local_coupling_condition} is the local equivalent of \eqref{eq:coupling_condition1}, we refer to it as the \emph{local coupling condition}. Finally, we can use the same approach as in \eqref{eq:trick_global_to_local} and assemble the coupled vectors $\bm{b}^\ell$ and $\bm{b}^{\ell-1}$ as
\begin{align}
\bm{b}^\ell = \sum\limits_{e=1}^n(L_e^\ell)^T\bm{b}^\ell_e,\quad \bm{b}^{\ell-1} = \sum\limits_{e=1}^n(L_e^{\ell-1})^T\bm{b}^{\ell-1}_e, 
\end{align}
where $n$ is the number of supermesh elements. This enforces the correct distribution since sums of Gaussian random variables are Gaussian and the covariance structure is correct. In particular,
\begin{equation}
\begin{split}
\E[\bm{b}^l(\bm{b}^l)^T] &= \sum\limits_{i,j=1}^n(L_i^l)^T\E[\bm{b}^l_i(\bm{b}^l_j)^T]L_j^l = \sum\limits_{i=1}^n(L_i^l)^T\E[\bm{b}^l_i(\bm{b}^l_i)^T]L_i^l\\
&=(L^l)^T\diag_i(M^l_i)L^l = M^l,\quad\text{for }l\in\{\ell,\ell-1\},
\end{split}
\end{equation}
and,
\begin{equation}
\begin{split}
\E[\bm{b}^\ell(\bm{b}^{\ell-1})^T] &= \sum\limits_{i,j=1}^n(L_i^\ell)^T\E[\bm{b}^\ell_i(\bm{b}^{\ell-1}_j)^T]L_j^{\ell-1} = \sum\limits_{i=1}^n(L_i^\ell)^T\E[\bm{b}^\ell_i(\bm{b}^{\ell-1}_i)^T]L_i^{\ell-1}\\
&=(L^\ell)^T\diag_i(M^{\ell,\ell-1}_i)L^{\ell-1} = M^{\ell,\ell-1},
\end{split}
\end{equation}
where we have used that $\bm{b}^l_i$ and $\bm{b}^l_j$ are independent for $i\neq j$ for $l\in\{\ell,\ell-1\}$ and that $\bm{b}^\ell_i$ is independent from $\bm{b}^{\ell-1}_j$ if $i\neq j$. Thus, again the global sampling problem can be recast as a series of much smaller, independent, local sampling problems.

Finally, it remains to devise a strategy for sampling realizations of
the local vectors $\bm{b}_e$ on a given supermesh element $e$. The
following result demonstrates that the covariance matrix of $\bm{b}_e$
is singular and how such sampling can be simplified.
\begin{lemma}
	\label{th:coupling_theorem}
	Let $V^{\ell}$ and $V^{\ell-1}$ be finite element spaces over two tessellations $D_h^\ell$, $D_h^{\ell-1}$ of the same domain. Let $S$ be a supermesh of $D_h^\ell$ and $D_h^{\ell-1}$. Let $\phi_i^\ell$, $\phi_j^{\ell-1}$ for $i=1,\dots,m_{e_\ell}$, $j=1,\dots,m_{e_{\ell-1}}$ be the basis functions of $V^\ell$ and $V^{\ell-1}$ respectively that have non-zero support over $e$. Let $V^{\ell}|_e = \textnormal{span}(\phi_1^{\ell}|_e,\dots,\phi^{\ell}_{m_{e_\ell}}|_e)$ and $V^{\ell-1}|_e=\textnormal{span}(\phi_1^{\ell-1}|_e,\dots,\phi^{\ell-1}_{m_{e_{\ell-1}}}|_e)$ be the restrictions of $V^\ell$ and $V^{\ell-1}$ to $e$. Assume that $V^{\ell-1}|_e \subseteq V^\ell|_e$, i.e.~that the restrictions are nested and that $M_e^{\ell}$ as defined in \eqref{eq:local_supermesh_mass_matrix} is non-singular, then
	\begin{align}
	\label{eq:coupling_trick}
	\textrm{rank}(M_e) = \textrm{rank}(M^\ell_e) \quad\text{and}\quad M^{\ell-1}_e = (M^{\ell,\ell-1}_e)^T(M^\ell_e)^{-1}M^{\ell,\ell-1}_e.
	\end{align}
\end{lemma}

\begin{proof}
	Since $V^{\ell-1}|_e \subseteq V^\ell|_e$, then we have that, for all $j$, $\phi_j^{\ell-1} \in V^\ell|_e$, which in turn means that there exist a set of coefficients $r_{ji}\in\R$ such that $\phi_j^{\ell-1}=\sum_ir_{ji}\phi_i^\ell$. Now, let $R_e$ be a $m_{e_\ell}\times m_{e_{\ell-1}}$ matrix such that $(R_e)_{i,j} = r_{ji}$, and define the vector functions
	\begin{align}
	\bm{\phi}^{\ell-1} = \left[\begin{array}{c}
	\phi_1^{\ell-1}\\
	\vdots\\
	\phi_{m_{e_{\ell-1}}}^{\ell-1}
	\end{array}\right],\quad
	\bm{\phi}^{\ell} = \left[\begin{array}{c}
	\phi_1^{\ell}\\
	\vdots\\
	\phi_{m_{e_{\ell}}}^{\ell-1}
	\end{array}\right].
	\end{align}
	We then have that
	\begin{align}
	\bm{\phi}^{\ell-1} = R^T_e\bm{\phi}^\ell.
	\end{align}
	
	This implies that we can now rewrite $M^{\ell-1}_e$ as
	\begin{align}
	\label{eq:proof_1}
	M^{\ell-1}_e = \int_e \bm{\phi}^{\ell-1}(\bm{\phi}^{\ell-1})^T \text{d}x = R^T_e \int_e \bm{\phi}^{\ell}(\bm{\phi}^{\ell})^T \text{d}x\ R_e = R^T_e M^\ell_e R_e,
	\end{align}
	since $R_e$ is constant. Similarly, for $M^{\ell,\ell-1}_e$ we have,
	\begin{align}
	\label{eq:proof_2}
	M^{\ell,\ell-1}_e = \int_e \bm{\phi}^{\ell}(\bm{\phi}^{\ell-1})^T \text{d}x = \int_e \bm{\phi}^{\ell}(\bm{\phi}^{\ell})^T \text{d}x\ R_e = M^\ell_e R_e.
	\end{align}
	
	Combining \eqref{eq:proof_1} and \eqref{eq:proof_2} with the assumption that $M^\ell_e$ is invertible thus yields the second equation in \eqref{eq:coupling_trick}, since
	\begin{align}
	(M^{\ell,\ell-1}_e)^T(M^\ell_e)^{-1}M^{\ell,\ell-1}_e = R_e^TM^\ell_e(M^\ell_e)^{-1}M^{\ell}_eR_e = R^T_e M^\ell_e R_e = M^{\ell-1}_e.
	\end{align}
	Pulling \eqref{eq:local_supermesh_mass_matrix}, \eqref{eq:proof_1} and \eqref{eq:proof_2} together we can now express $M_e$ as,
	\begin{align}
	\label{eq:proof_3}
	M_e = \left[
	\begin{array}{c|c}%
	M^{\ell}_e & M^\ell_e R_e \\\hline\\[-1em]
	R_e^TM^\ell_e & R^T_e M^\ell_e R_e
	\end{array}
	\right] = \left[
	\begin{array}{c|c}%
	I & 0 \\\hline\\[-1em]
	R_e^T & I
	\end{array}
	\right]
	\left[
	\begin{array}{c|c}%
	M^\ell_e & 0 \\\hline\\[-1em]
	0 & 0
	\end{array}
	\right]
	\left[
	\begin{array}{c|c}%
	I & R_e \\\hline\\[-1em]
	0 & I
	\end{array}
	\right],
	\end{align}
	where we have used the fact that $M^\ell_e$ is symmetric. Since $M_e$ is symmetric and the two block triangular matrices on the right hand side of \eqref{eq:proof_3} are invertible, Sylvester's law of inertia \cite{Strang2016} gives that
	\begin{align}
	\label{eq:proof_4}
	\textnormal{rank}(M_e) = \textnormal{rank}\left(\left[
	\begin{array}{c|c}%
	M^\ell_e & 0 \\\hline\\[-1em]
	0 & 0
	\end{array}
	\right]\right) = \textnormal{rank}(M^\ell_e),
	\end{align}
	which concludes the proof.
\end{proof}
The assumptions of Lemma \ref{th:coupling_theorem} are mild and are satisfied by most finite element families e.g.~Lagrange elements (continuous piecewise polynomials defined relative to the tessellations).

Using Lemma \ref{th:coupling_theorem} we can now sample $\bm{b}_e=[(\bm{b}^\ell_e)^T, (\bm{b}^{\ell-1}_e)^T]^T$ by enforcing the local coupling condition \eqref{eq:local_coupling_condition} as follows. For each supermesh element $e$:
\begin{enumerate}[leftmargin=1cm]
	\item Compute $M^{\ell,\ell-1}_e$ and the Cholesky factorization $M^\ell_e = H_eH_e^T$. 
	\item Sample $\bm{z}_e\sim\mathcal{N}(0,I)$ of length $m_{e_\ell}$ and set $\bm{b}^\ell_e = H_e\bm{z}_e$.
	\item Compute $\bm{b}^{\ell-1}_e$ as $\bm{b}^{\ell-1}_e = (M^{\ell,\ell-1})^TH_e^{-T}\bm{z}_e$.
\end{enumerate}
Note that the $\bm{b}^\ell_e$ and $\bm{b}^{\ell-1}_e$ sampled this way satisfy the local coupling condition \eqref{eq:local_coupling_condition} since, by \eqref{eq:coupling_trick} and the fact that $H_e$ is the Cholesky factor of $M^\ell_e$, we have that
\begin{equation}
\E[\bm{b}^\ell_e(\bm{b}^\ell_e)^T] = H_e\E[\bm{z}_e\bm{z}_e^T]H_e^T = M^\ell_e,
\end{equation}
second,
\begin{equation}
\begin{split}
\E[\bm{b}^{\ell-1}_e(\bm{b}^{\ell-1}_e)^T] &= (M^{\ell,\ell-1}_e)^TH_e^{-T}\E[\bm{z}_e\bm{z}_e^T]H_e^{-1}M^{\ell,\ell-1}_e\\
&= (M^{\ell,\ell-1}_e)^T(H_eH_e^T)^{-1}M^{\ell,\ell-1}_e\\
&= (M^{\ell,\ell-1}_e)^T(M^\ell_e)^{-1}M^{\ell,\ell-1}_e = M^{\ell-1}_e,
\end{split}
\end{equation}
and third,
\begin{equation}
\E[\bm{b}^{\ell}_e(\bm{b}^{\ell-1}_e)^T] = H_e\E[\bm{z}_e\bm{z}_e^T]H_e^{-1}M^{\ell,\ell-1}_e = H_eH_e^{-1}M^{\ell,\ell-1}_e = M^{\ell,\ell-1}_e.
\end{equation}

In the case in which the transformation to the reference element is affine (such as with Lagrange elements on simplices) the sampling can be made more efficient by sampling white noise directly on the supermesh and then interpolating it onto the parent mashes. This strategy exploits the following result.
\begin{corollary}[of Lemma \ref{th:coupling_theorem}]
	\label{coroll:lemma}
	Let $V^{\ell}$ and $V^{\ell-1}$ be FEM approximation subspaces over two triangulations $D_h^\ell$, $D_h^{\ell-1}$ of the same domain. Let $S$ be a supermesh of $D_h^\ell$ and $D_h^{\ell-1}$ and let $V^S$ be a FEM approximation subspace over $S$. With the same notation as in Lemma \ref{th:coupling_theorem}, for each supermesh element $e$ let $V^S|_e$, $V^{\ell}|_e$ and $V^{\ell-1}_e$ be the restrictions of $V^S$, $V^\ell$ and $V^{\ell-1}$ to $e$. Let $M^S_e$ be the local mass matrix over $V^S|_e$. Assume that $V^{\ell}|_e, V^{\ell-1}|_e  \subseteq V^S|_e$, i.e.\ that the parent mesh restrictions are nested within the supermesh restriction. Then there exist local interpolation matrices $(R_e^\ell)^T$ and $(R_e^{\ell-1})^T$ such that
	\begin{align}
	\label{eq:supermesh_nested_relations}
	M_e^\ell = (R_e^\ell)^TM^S_eR_e^\ell,\quad M_e^{\ell-1} = (R_e^{\ell-1})^TM^S_eR_e^{\ell-1},\quad M_e^{\ell,\ell-1} = (R_e^\ell)^TM^S_eR_e^{\ell-1}.
	\end{align}
\end{corollary}
\begin{proof}
	Let $l\in\{\ell,\ell-1\}$. The proof for the first two equations in \eqref{eq:supermesh_nested_relations} follows from the first part of the proof of Lemma \ref{th:coupling_theorem} by replacing $\ell-1$ with $l$ and $\ell$ with $S$. This argument gives us that
	\begin{align}
	\bm{\phi}^l = (R_e^l)^T\bm{\phi}^S,
	\end{align}
	from which we also obtain the last relation in \eqref{eq:supermesh_nested_relations} since
	\begin{align}
	M^{\ell,\ell-1}_e = \int_e \bm{\phi}^{\ell}(\bm{\phi}^{\ell-1})^T \text{d}x = (R_e^{\ell})^T\int_e \bm{\phi}^{S}(\bm{\phi}^{S})^T \text{d}x\ R_e^{\ell-1} = (R_e^{\ell})^TM^S_e R_e^{\ell-1}.
	\end{align}
\end{proof}
By using this result and the strategy highlighted in Remark \ref{remark:Lagrange}, we can sample $\bm{b}^l_e$ for $l\in\{\ell,\ell-1\}$ by computing
\begin{align}
\bm{b}^l_e = (R_e^l)^TH_{r}|e_r|^{-1/2}|e|^{1/2}\bm{z}_e,\quad\text{with}\quad \bm{z}_e\sim\mathcal{N}(0,I),
\end{align}
since Remark \ref{remark:Lagrange} yields the relation $M^S_e/|e| = H_rH^T_r/|e_r|=\text{ const}$, where $H_r$ is the Cholesky factor of the local mass matrix over the reference element $e_r$. Note that $H_r$ has to be computed only once. The advantage of performing this operation is that it avoids the assembly and factorization of each supermesh element local mass matrix.

\begin{remark}[Simpler cases: nested meshes and $p$-refinement]
	In the case in which the meshes of the MLMC hierarchy are nested, everything discussed is still valid by taking the supermesh to be the finer of the two meshes that define the MLMC level. In the case in which the MLMC hierarchy is constructed by using $p$-refinement there is only one mesh in the hierarchy and everything still applies by taking this mesh to be the `supermesh'. In both cases a supermesh construction is not required in practice.
\end{remark}
\begin{remark}
	The coupling approach presented can also be used to couple the same white noise sample over the whole hierarchy of meshes. This enables the use of geometric full-multigrid \cite{trottenberg2000multigrid} to solve the problem given by \eqref{eq:coupled1}--\eqref{eq:coupled2} on the finer grid with optimal multigrid complexity. 
\end{remark}

\section{Numerical results}
\label{sec:num_results}
In this section we investigate the performance of the techniques presented. We consider the following PDE:
\begin{align}
\label{eq:PDE_of_interest}
\begin{array}{rclcr}
-\nabla\cdot(e^{u(x,\omega)}\nabla q(x,\omega)) = 1, & & x\in G = (-0.5,0.5)^d, & & \omega\in\Omega,\\
q(x,\omega) = 0,& & x\in \partial G,& &\omega\in\Omega,
\end{array}
\end{align}
where $u$ is a Mat\'ern field as given by \eqref{eq:Matern} with mean and variance chosen so that $e^{u(x,\omega)}$ has mean $1$ and standard deviation $0.2$. We choose $D=(-1,1)^d$ as the outer computational domain on which to solve \eqref{eq:truncated_white_noise_eqn}. The output functional of interest we consider here is the $L^2(G)$ norm of $q$ squared, namely ${P(\omega)=\Vert q\Vert ^2_{L^2(G)}(\omega)}$.

We solve \eqref{eq:truncated_white_noise_eqn} and \eqref{eq:PDE_of_interest} with the FEniCS software package \cite{LoggMardalEtAl2012a} and we discretize the two problems by using continuous Lagrange finite elements of the same degree. For the linear solver, we use the BoomerAMG algebraic multigrid algorithm from Hypre \cite{hypre} as a preconditioner and the conjugate gradient routine of PETSc \cite{balay2014petsc} for all equations. As convergence criterion for the solver we require that the absolute size of the preconditioned residual norm is below a tolerance of $10^{-15}$. We use the libsupermesh software package \cite{libsupermesh-tech-report} for the supermesh constructions. 

When using $h$-refinement, we construct the MLMC mesh hierarchies $\{D_h^\ell\}_{\ell=1}^L$ and $\{G_h^\ell\}_{\ell=1}^L$ in such a way that $G_h^\ell$ is embedded within $D_h^\ell$ for all $\ell$, but $D_h^{\ell-1}$ and $G_h^{\ell-1}$ are not nested respectively within $D_h^{\ell}$ and $G_h^{\ell}$ for all $\ell>1$. As the meshes are non-nested, a supermesh construction is required to couple each MLMC level. The mesh hierarchies we use are composed of $L=9$ meshes in 2D and $L=5$ meshes in 3D. The coarsest mesh in each hierarchy is uniform, while the other meshes are non-uniform and unstructured. Since the convergence behavior of the FEM is dependent on the quality of the mesh used, we try to sanitize our numerical results from this effect by choosing meshes whose quality indicators do not vary excessively throughout the hierarchies. Basic properties of the different meshes and number of elements of the constructed supermeshes are summarized in Tables \ref{tab:table1} and \ref{tab:table2}. Note that the number of elements in the supermesh is in practice always bounded by a constant times the number of elements of the finer parent mesh. This constant is dimension-dependent, and larger in 3D than 2D (cf.~Table~\ref{tab:table2}).

\begin{table}[h!]
	\centering
	\label{tab:table1}
	\begin{tabular}{c|l|r|c|c}
		\toprule
		$\ell$ (2D) & \multicolumn{1}{c|}{$h_\ell$} & \multicolumn{1}{c|}{$n_\ell$} & $(RR_{\min}, RR_{\max})$ & $n_{S_\ell}/n_\ell$ \\ \midrule
		$1$                               & $0.707$                       & $32$                         & $(0.83,\ 0.83)$         & n/a \\ 
		$2$                               & $0.416$                       & $120$                        & $(0.61,\ 1)$            & $2.03$\\ 
		$3$                               & $0.194$                       & $500$                        & $(0.61,\ 1)$            & $2.32$\\ 
		$4$                               & $0.098$                       & $2106$                       & $(0.55,\ 1)$            & $2.45$\\ 
		$5$                               & $0.049$                       & $8468$                       & $(0.45,\ 1)$            & $2.44$\\ 
		$6$                               & $0.024$                       & $33686$                      & $(0.46,\ 1)$            & $2.46$\\ 
		$7$                               & $0.012$                       & $134170$                     & $(0.41,\ 1)$            & $2.46$\\ 
		$8$                               & $0.006$                       & $535350$                     & $(0.42,\ 1)$            & $2.46$\\ 
		$9$                               & $0.003$                       & $2143162$                    & $(0.42,\ 1)$            & $2.47$\\ \bottomrule
	\end{tabular}
	\caption{Properties of the 2D mesh hierarchy: mesh level $l$, maximal element size $h_\ell$, number of elements $n_\ell$, minimal and maximal element radius ratios $RR_{\min}$ and $RR_{\max}$, and the number of elements of the supermesh constructed using the meshes on levels $\ell$ and $\ell-1$ as parent meshes $n_{S_\ell}$. $RR$ is computed as $d\times r^e_{in}/r^e_{circ}$, where $r^e_{in}$ and $r^e_{circ}$ are the in-radius and the circumradius of element $e$ respectively. Note that the element size roughly decreases proportional to $2^{-\ell}$.}
\end{table}
\begin{table}[h!]
	\centering
	\begin{tabular}{c|l|r|c|c|c}
		\toprule
		$\ell$ (3D) & \multicolumn{1}{c|}{$h_\ell$} & \multicolumn{1}{c|}{$n_\ell$} & $(RR_{\min}, RR_{\max})$ & $(DA_{\min}, DA_{\max})$ & $n_{S_\ell}/n_\ell$ \\ \midrule
		$1$                               & $0.866$                       & $384$                        & $(0.72,\ 0.72)$          & $(0.79,\ 1.57)$         & n/a\\ 
		$2$                               & $0.437$                       & $7141$                       & $(0.22,\ 1)$             & $(0.21,\ 2.82)$         & $17$\\ 
		$3$                               & $0.280$                       & $22616$                      & $(0.18,\ 1)$             & $(0.21,\ 2.83)$         & $66$\\ 
		$4$                               & $0.138$                       & $190081$                     & $(0.13,\ 1)$             & $(0.21,\ 2.85)$         & $42$\\ 
		$5$                               & $0.070$                       & $1519884$                    & $(0.12,\,1)$             & $(0.21,\ 2.85)$         & $46$\\ \bottomrule
	\end{tabular}
	\caption{Properties of the 3D mesh hierarchy: mesh level $l$, maximal element size $h_\ell$, number of elements $n_\ell$, minimal and maximal element radius ratios $RR_{\min}$ and $RR_{\max}$, the minimum and maximum element dihedral angles $DA_{\min}$ and $DA_{\max}$ respectively, and the number of elements of the supermesh constructed using the meshes on levels $\ell$ and $\ell-1$ as parent meshes. Note that the element size of the last three levels roughly decreases proportional to $2^{-\ell}$.}
	\label{tab:table2}
\end{table}

When using $p$-refinement, we define the MLMC levels by taking the coarsest mesh in the 2D hierarchy and by increasing the polynomial degree of the FEM subspaces linearly so that $p_\ell = \ell$ for $\ell=1,\dots,L$, with $L=9$. We do not consider $p$-refinement in the 3D case as it does not offer any additional complications other than an increased computational cost.

\subsection{Mat\'ern field convergence}
\label{sec:matern_field_conv}
We first address the convergence of the solution of \eqref{eq:truncated_white_noise_eqn} to the Mat\'ern field of interest. In practice, the exact solution of \eqref{eq:truncated_white_noise_eqn} is not known, so we consider the coupled equations \eqref{eq:coupled1} and \eqref{eq:coupled2} instead. 
We monitor the quantities
\begin{align}
\label{eq:error_measures}
\left|\E\left[\Vert u\Vert _{L^2(G)}^2 - \Vert u_{\ell-1}\Vert _{L^2(G)}^2\right]\right|,\quad \V\left[\Vert u_\ell\Vert _{L^2(G)}^2 - \Vert u_{\ell-1}\Vert _{L^2(G)}^2\right].
\end{align}  
Note that the value of $\E[\Vert u\Vert _{L^2(G)}^2]$ is known up to the error introduced by truncating $\R^d$ to $D$ since we can exchange the order of expectation and integration:
\begin{align}
\E\left[\Vert u\Vert _{L^2(G)}^2\right] = \E\left[\int_G u^2 \text{ d}x\right] = \int_G \E[u^2]\text{ d}x \approx \sigma^2 |G|,
\end{align}
where we have used the fact that $\E[u^2] \approx \sigma^2$ for all $x\in G$ (the relation only holds approximately due to domain truncation error).

The following result by Bolin et al.~\cite{Bolin2017} establishes a theoretical estimate for the expected convergence rates.
\begin{theorem}[Theorem 2.10 and Corollary 2.4 in \cite{Bolin2017}]
	\label{th:Bolin_convergence}
	Let $u$ be the solution of \eqref{eq:truncated_white_noise_eqn} ($k=1$ case) and let $u_h$ be its FEM approximation obtained by using continuous Lagrange elements over a mesh of maximum element size $h$. Then there exist constants $c_1$ and $c_2$ such that
	\begin{align}
	\E\left[\Vert u - u_h\Vert _{L^2(D)}^2\right]^{1/2} \leq c_1h^{2 - d/2},\\
	\left|\E\left[\Vert u\Vert _{L^2(D)}^2 - \Vert u_h\Vert _{L^2(D)}^2\right]\right| \leq c_2 h^{4 - d}.
	\end{align}
\end{theorem}

Note that the norms appearing in the error estimates of Theorem \ref{th:Bolin_convergence} refer to the outer domain $D$, while the norms we consider in \eqref{eq:error_measures} are taken over the inner domain $G$. As $G\subset D$, we expect to observe the same convergence behavior.
We are not aware of any error estimates in the literature for the variance in \eqref{eq:error_measures}, but the convergence order observed in practice is usually twice that of the expectation (see for example \cite{Cliffe2011}), provided that the polynomial degree of the FEM basis is sufficiently high.

We consider the convergence behavior of the FEM approximation of the solution of \eqref{eq:white_noise_eqn} in the $h$-refinement case with the sampling strategy described in Section \ref{sec:white_noise_sampling}. We fix $\lambda = 0.2$ and we consider Mat\'ern fields of smoothness $\nu = 1$, $\nu = 3$ ($k=1$ and $k=2$ respectively in 2D) and $\nu=1/2$ ($k=1$ in 3D). For the $\nu=1$ and $\nu=1/2$ cases we use continuous piecewise linear (P1) elements, while for the $\nu=3$ case we use continuous piecewise quadratic (P2) elements.

Since each sample drawn by solving \eqref{eq:white_noise_eqn} is computationally expensive we are unable to take large numbers of samples as is generally done for 1D stochastic differential equations \cite{giles2015multilevel}. We therefore take $N_\ell = 5000$ Monte Carlo samples on all levels in 2D and $N_\ell = 1000$ samples in 3D. To verify that these numbers of samples are sufficient for accurate representation, we compute approximate $99.73\%$ confidence intervals (CIs) for all the quantities of interest as $3\bar{\sigma}_\ell/\sqrt{N_\ell}$, where $\bar{\sigma}_\ell$ is the sample standard deviation of the output functional of interest on level $\ell$. In all cases considered here but one, the FEM error dominates and the confidence intervals are negligibly small (so small that they would not be visible on the convergence plots). The relatively small number of samples only becomes a problem in the $\nu=3$ case where the FEM convergence is much faster and the Monte Carlo error dominates. In this case we replace the $\Vert u\Vert _{L^2(G)}$ term in the expectation in \eqref{eq:error_measures} with $\Vert u_\ell\Vert _{L^2(G)}$, and we instead monitor the convergence of the following quantity,
\begin{align}
\left|\E\left[\Vert u_{\ell}\Vert _{L^2(G)}^2 - \Vert u_{\ell-1}\Vert _{L^2(G)}^2\right]\right|.
\end{align}
The advantage of doing this is that the variance of this error measure decreases with the level (see Figure \ref{fig:matern2D}) and $5000$ samples are enough to obtain good accuracy.

Results are shown in Figures \ref{fig:matern2D} (2D) and \ref{fig:matern3D} (3D). For both the 2D and 3D experiments, we observe the theoretically predicted convergence rates in terms of the mesh size (after a pre-asymptotic regime). However, we note how convergence is less regular than expected (especially in the 3D case) because of the unstructured meshes employed. This behavior does not appear when uniform meshes are used (not shown). Apart from the $\nu=3$ case, the convergence order of the variance seems to be twice the convergence order of the expectation. In the $\nu=3$ case we observe order $6$ for the variance with P2 elements (Figure \ref{fig:matern2D}) and order $8$ with P3 elements (not shown). We conjecture that the variance convergence order is bounded by $2(p+1)$, where $p$ is the polynomial degree of the FEM basis functions.

\begin{figure}[h!] 
	\centering
	\includegraphics[trim=3cm 0cm 4cm 1cm, clip=true, width=.8\textwidth]{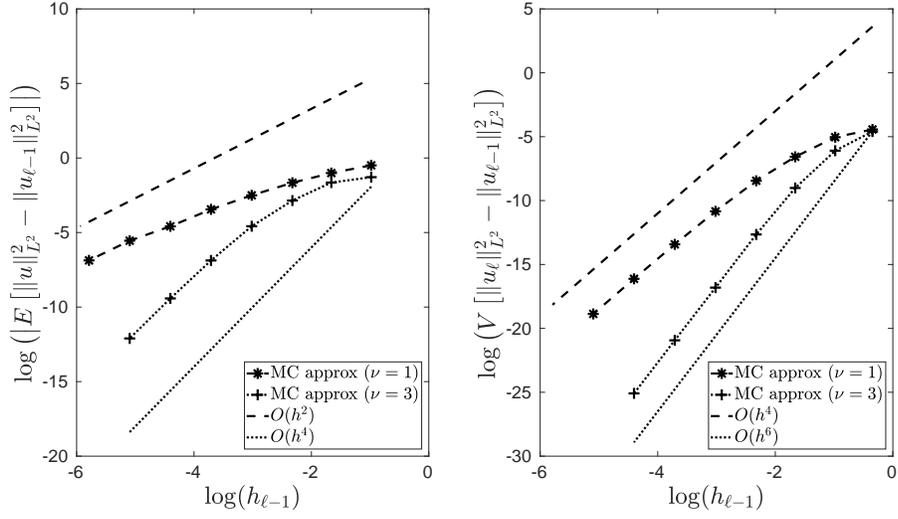}
	\centering
	\caption{Convergence behavior of the FEM approximation to \eqref{eq:white_noise_eqn} with $h$-refinement in 2D. Plots show (the natural logarithm of) the expected value $\E$ (left) and variance $\V$ (right) versus maximal mesh size $h_l$. For each level $l$, the fields $u_\ell$ and $u_{\ell-1}$ have been sampled by coupling white noise realizations as described in Section \ref{sec:white_noise_sampling}. As mentioned in the text, to compute the expected value in the $\nu=3$ case we have replaced $\Vert u\Vert_{L^2(G)}$ with $\Vert u_\ell\Vert_{L^2(G)}$.}
	\label{fig:matern2D}
\end{figure}
\begin{figure}[h!]
	\centering
	\includegraphics[trim=3cm 0cm 4cm 1cm, clip=true, width=.8\textwidth]{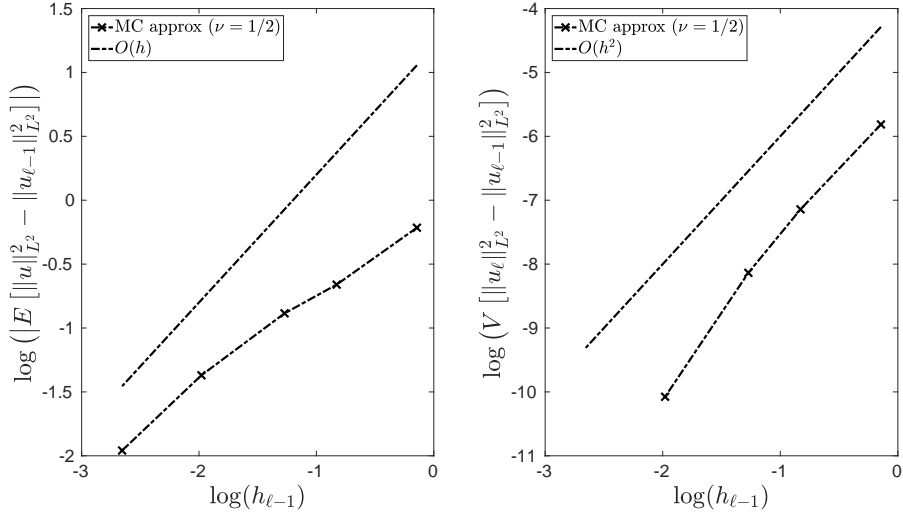}
	\centering
	\caption{Convergence behavior of the FEM approximation to \eqref{eq:white_noise_eqn} with $h$-refinement in 3D. Plots show (the natural logarithm of) the expected value $\E$ (left) and variance $\V$ (right) versus maximal mesh size $h_l$. The fields $u_\ell$ and $u_{\ell-1}$ have been sampled by coupling white noise realizations as described in Section \ref{sec:white_noise_sampling}.}
	\label{fig:matern3D}
\end{figure}

\begin{figure}[h!]
	\centering
	\includegraphics[trim=3cm 0cm 4cm 0cm, clip=true, width=.8\textwidth]{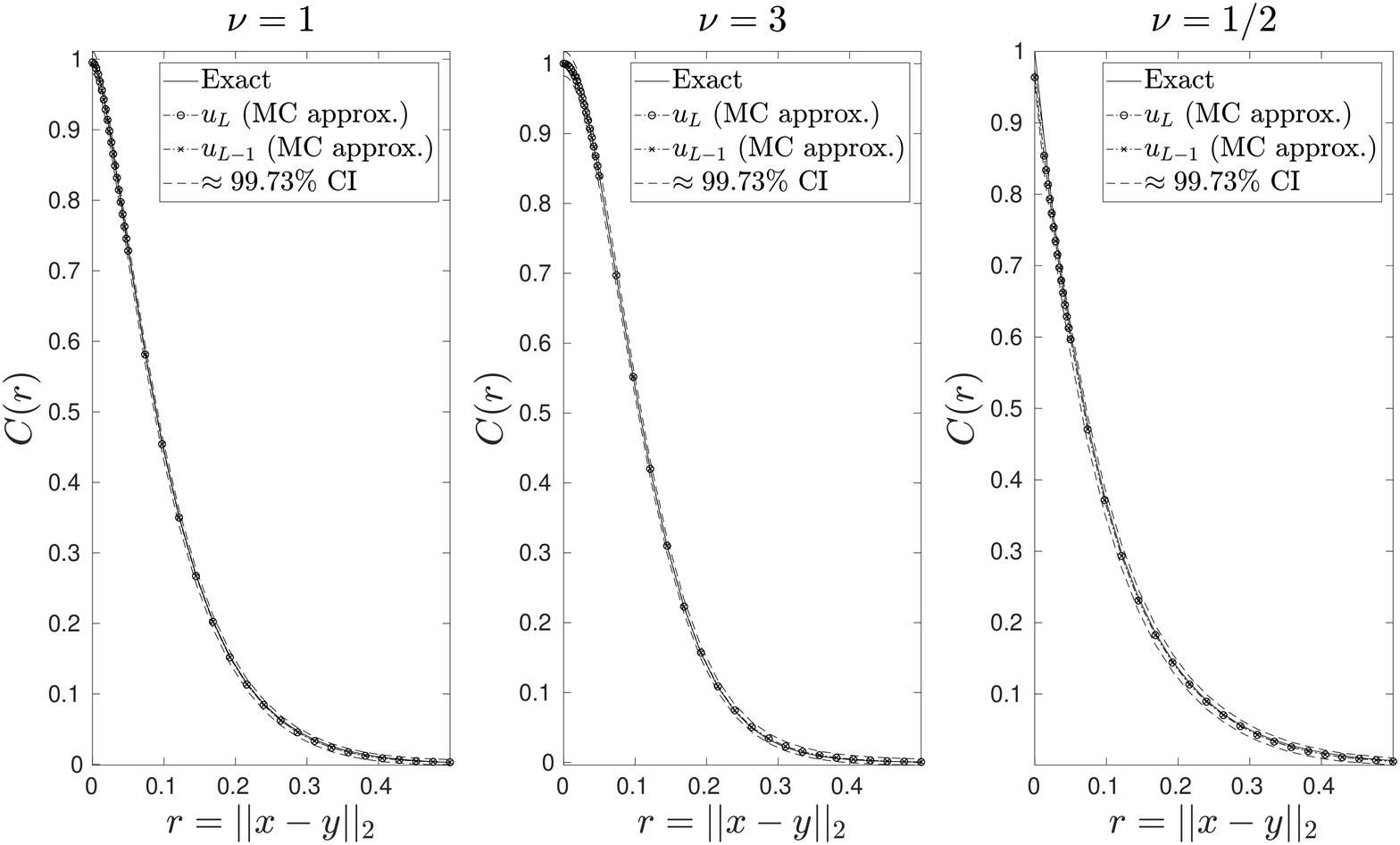}
	\centering
	\caption{Plot of exact covariances and sample covariances vs distance $r$ of the FEM solutions of \eqref{eq:coupled1} and \eqref{eq:coupled2} for three different values of $\nu$ in the $h$-refinement case. The exact covariance $C(r)$ is given by \eqref{eq:Matern}. For the $\nu=3$ case, an extra elliptic PDE solve is needed, see \eqref{eq:iterative_white_noise_eqn}.}
	\label{fig:covariance_convergence}
\end{figure}

In Figure \ref{fig:covariance_convergence}, we demonstrate how the Mat\'ern covariance and the field coupling are correctly enforced by our technique. We compare the covariances of the coupled Mat\'ern fields obtained by solving \eqref{eq:coupled1} and \eqref{eq:coupled2} on the finest level of the MLMC hierarchy with the exact Mat\'ern covariance given by \eqref{eq:Matern}. The estimated covariances match each other and the exact expression closely, demonstrating that our coupling technique is accurate also in practice.

As a final verification step, we check that the coupled fields are consistent with the telescoping sum in \eqref{eq:telescoping_sum} and \eqref{eq:MLMC_estimator}, i.e.\ if we let $a$, $b$, $c$ be the MC approximations of $\E[\Vert u_\ell\Vert ^2_{L^2(G)} - \Vert u_{\ell-1}\Vert ^2_{L^2(G)}]$, $\E[\Vert u_\ell\Vert ^2_{L^2(G)}]$ and $\E[\Vert u_{\ell-1}\Vert ^2_{L^2(G)}]$ respectively, we aim to verify that
\begin{align}
a - b + c \approx 0,
\end{align}
at least to within the Monte Carlo accuracy. In Figure \ref{fig:telescoping_sum}, we plot the quantity
\begin{align}
\label{eq:def:T}
T(a, b, c) \equiv \frac{|a - b + c|}{3(\sqrt{\V_a} + \sqrt{\V_b} + \sqrt{\V_c})},
\end{align}
for different levels and Mat\'ern smoothness parameters $\nu$, where $\V_a$, $\V_b$ and $\V_c$ are the Monte Carlo approximations of the variances of $a$, $b$ and $c$. The probability of this ratio $T$ being greater than $1$ is less than $0.3\%$ (for further details, see \cite{giles2015multilevel}). We observe that $T$ ranges between $0$ and $0.4$ for the levels and smoothness parameters tested (Figure~\ref{fig:telescoping_sum}), and in particular is well below $1$. This indicates that our implementation of the MLMC algorithm correctly satisfies the telescoping summation formulation.

\begin{figure}[h!]
	\centering
	\includegraphics[trim=3cm 0cm 4cm 1cm, clip=true, width=.8\textwidth]{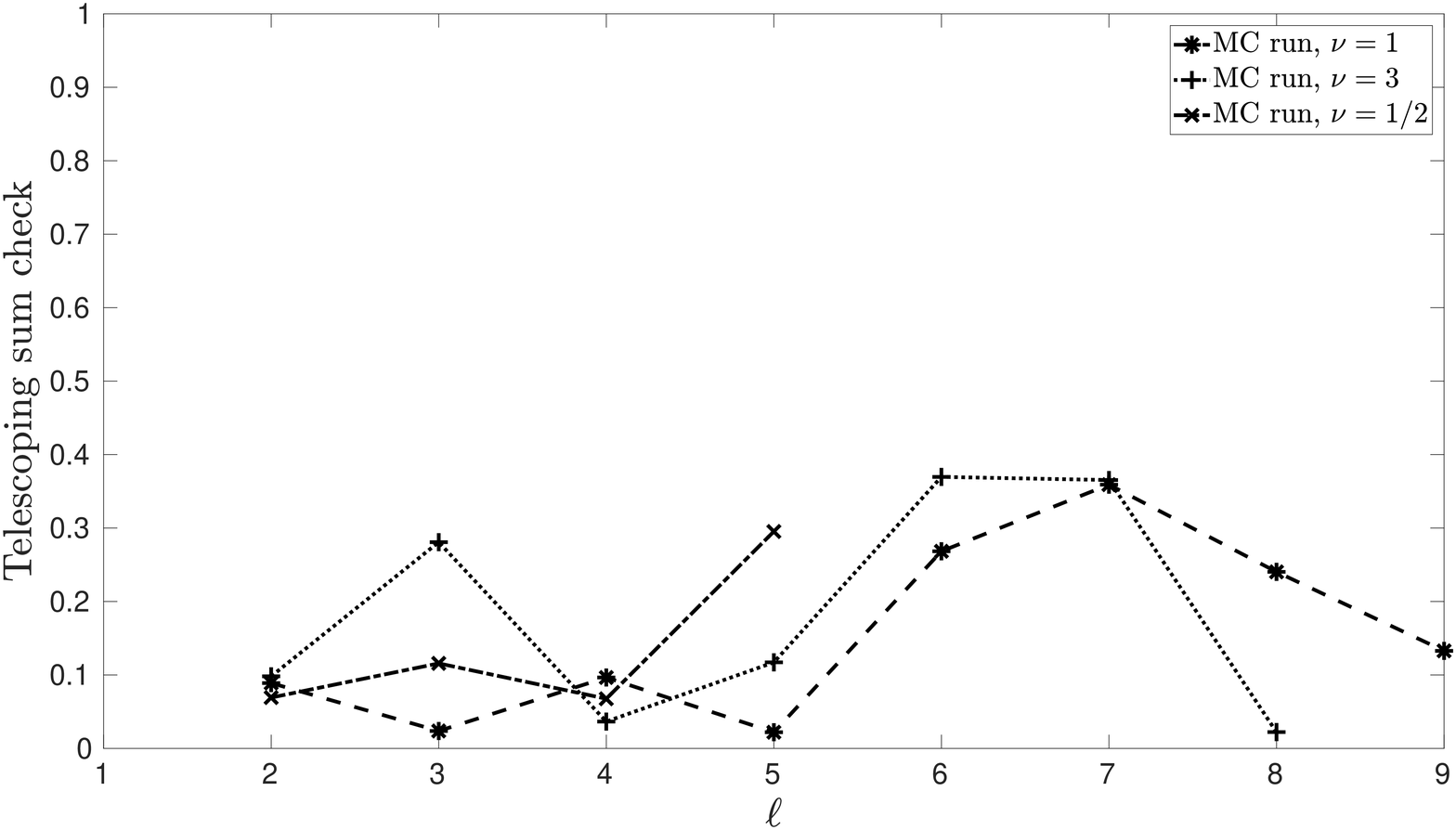}
	\centering
	\caption{Telescoping sum consistency check. Plot of $T(a, b, c)$ as defined by~\eqref{eq:def:T} versus level $l$ for ${a = \E[\Vert u_\ell\Vert ^2_{L^2(G)} - \Vert u_{\ell-1}\Vert ^2_{L^2(G)}]}$, $b = \E[\Vert u_\ell\Vert ^2_{L^2(G)}]$ and $c = \E[\Vert u_{\ell-1}\Vert ^2_{L^2(G)}]$ for different smoothness parameters $\nu$.}
	\label{fig:telescoping_sum}
\end{figure}

\subsection{MLMC convergence}
We now consider the convergence of the multilevel Monte Carlo method applied to \eqref{eq:PDE_of_interest}. In the case where $u$ is sampled exactly, the assumptions of the MLMC convergence theorem (Theorem \ref{th:MLMC_convergence}) hold for the $h$-refinement case with constants $\alpha = 2$ and $\beta = 4$ \cite{CharrierMLMC2013}. Furthermore, since we use multigrid to solve \eqref{eq:PDE_of_interest} and \eqref{eq:truncated_white_noise_eqn} we have $\gamma = d$. In the case where $\nu > 1$, the Mat\'ern field smoothness increases \cite{PetterAbrahamsen1997} and we expect higher convergence rates for the solution of \eqref{eq:PDE_of_interest}. For integer $\nu$ and exact sampling of $u$, if the domain $G$ is of class $C^{\nu+1}$, then the MLMC parameter values are given by $\alpha = \min(\nu+1, p+1)$ and $\beta=2\alpha$, where $p$ is the polynomial degree of the Lagrange elements used \cite{LiLiu2016}.

In our case, $u$ is approximated with the FEM and this could affect convergence. To verify that this is not what happens in practice, we first solve \eqref{eq:white_noise_eqn} with FEM for the same parameter values as in Subsection \ref{sec:matern_field_conv}, namely $\lambda=0.2$, $\nu=1$ and $\nu=3$ ($k=1$ and $k=2$ respectively in 2D) and $\nu=1/2$ ($k=1$ in 3D) using P1 elements for $\nu=1/2$ and $\nu=1$ and P2 elements for $\nu=3$. We then use the approximated Mat\'ern fields computed this way as coefficients in \eqref{eq:PDE_of_interest}, which we solve again using the same choice of finite elements.

Results are shown in Figures \ref{fig:darcy2D} and \ref{fig:darcy3D}. We observe that the convergence is unaffected by the approximation of the Mat\'ern fields and that the estimated convergence orders agree with the theory \cite{CharrierMLMC2013} apart from some discrepancies in the 3D case. This irregular behavior is probably due to the non-uniformity of the hierarchy (as we see from Table \ref{tab:table2}, the quality of the 3D meshes decreases with the level). This issue does not arise if the same numerical experiment is performed using a uniform hierarchy instead (hierarchy mesh sizes given by $h_\ell = 1.732\times 2^{-\ell}$, see Figure \ref{fig:darcy3Duniform}).

\begin{figure}[h!]
	\centering
	\includegraphics[trim=3cm 0cm 4cm 1cm, clip=true, width=.8\textwidth]{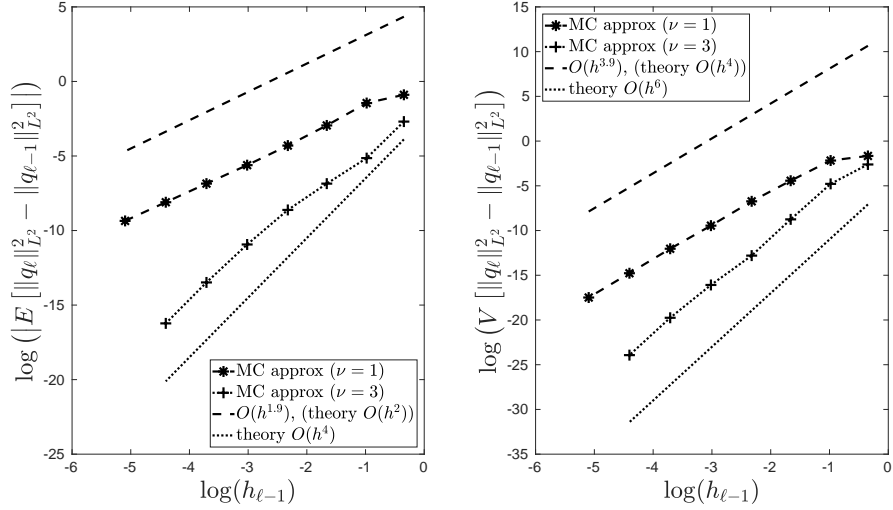}
	\centering
	\caption{Convergence behavior of the FEM approximation to the solution of \eqref{eq:PDE_of_interest} with $h$-refinement in 2D. The estimated convergence orders agree with the theory \cite{CharrierMLMC2013,LiLiu2016}.}
	\label{fig:darcy2D}
\end{figure}
\begin{figure}[h!]
	\centering
	\includegraphics[trim=3cm 0cm 4cm 1cm, clip=true, width=.8\textwidth]{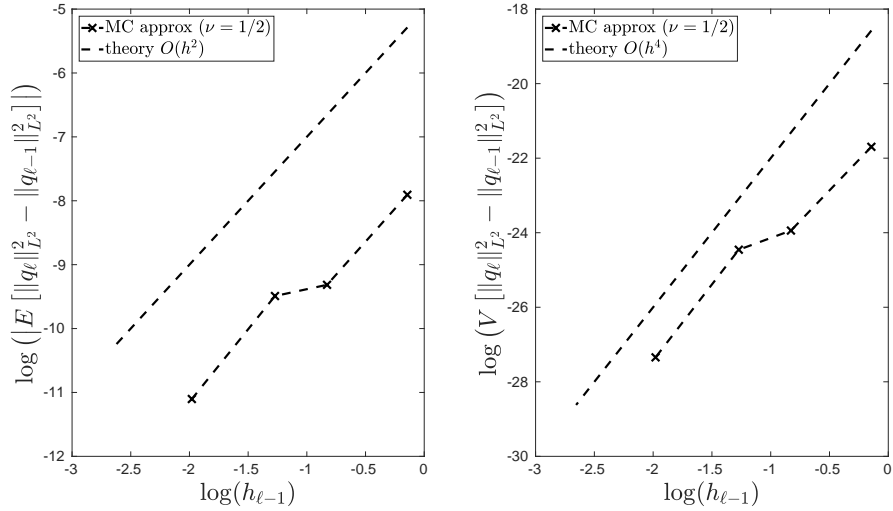}
	\centering
	\caption{Convergence behavior of the FEM approximation to the solution of \eqref{eq:PDE_of_interest} with $h$-refinement in 3D.}
	\label{fig:darcy3D}
\end{figure}
\begin{figure}[h!]
	\centering
	\includegraphics[trim=3cm 0cm 4cm 1cm, clip=true, width=.8\textwidth]{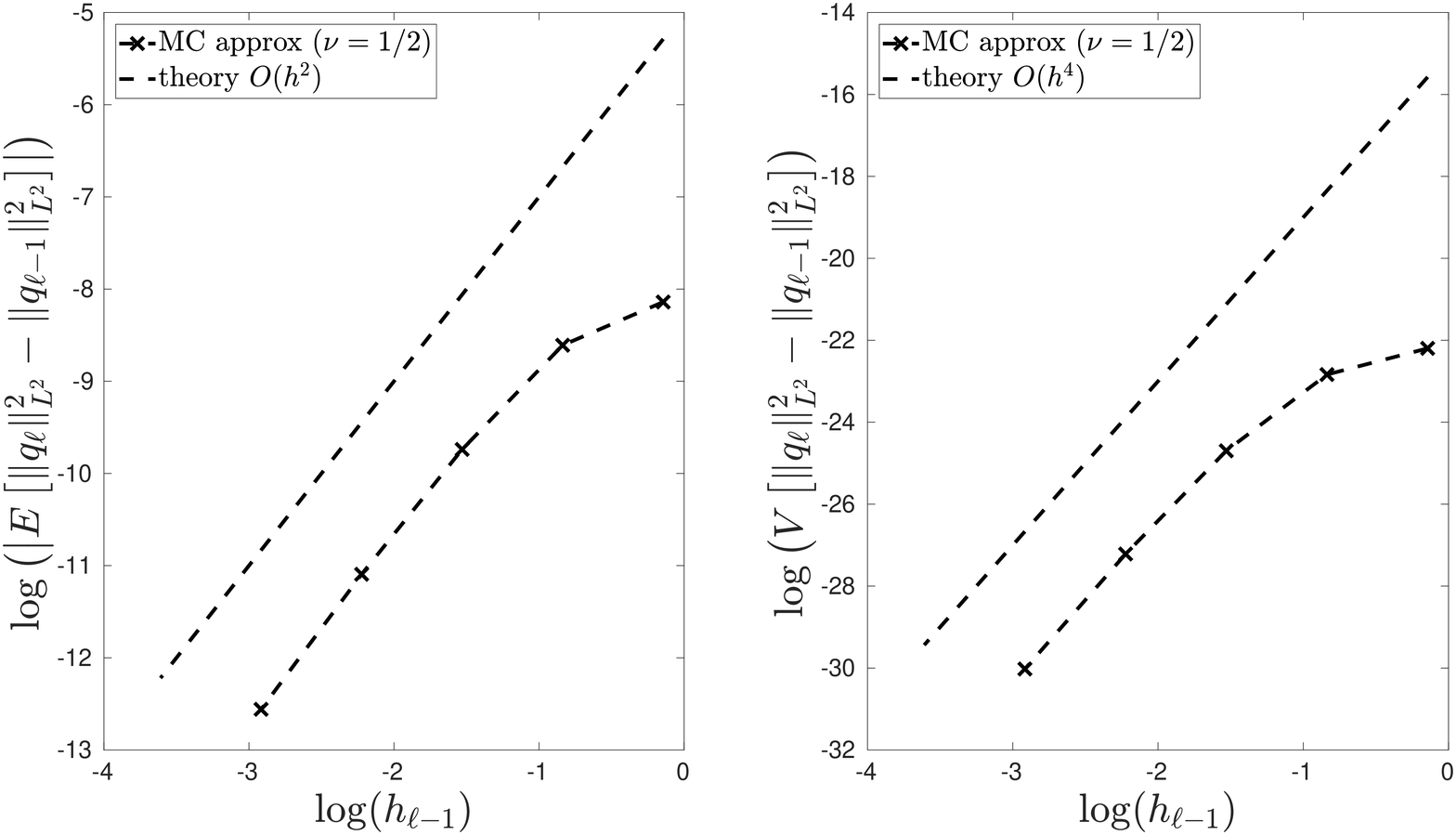}
	\centering
	\caption{Convergence behavior of the FEM approximation to the solution of \eqref{eq:PDE_of_interest} with $h$-refinement in 3D using a hierarchy of uniform meshes. The estimated convergence orders agree with the theory \cite{CharrierMLMC2013}. The mesh sizes are given by $h_\ell = 1.732\times 2^{-\ell}$.}
	\label{fig:darcy3Duniform}
\end{figure}

We now investigate how MLMC performs in practice. We use standard Monte Carlo and MLMC to estimate $\E[\Vert q\Vert_{L^2(G)}]$ at the same accuracy for $\nu=1$ (2D), P1 elements and for different error tolerances $\varepsilon$ (cf. \eqref{eq:MSE}). Again, the coefficient $u$ of \eqref{eq:PDE_of_interest} is also approximated with the FEM. We keep track of the total computational cost $C_{\text{tot}}$ and, in the MLMC case, of the number of samples $N_\ell$ taken on each level.

Results are shown in Figure \ref{fig:mlmc_convergence}. We observe that the number of levels used increases as the
tolerance $\varepsilon$ decreases (Figure \ref{fig:mlmc_convergence},
left). This behaviour reflects the targeted weak error accuracy
\cite{giles2015multilevel}: the number of samples is chosen by the
MLMC algorithm so as to optimize the total computational effort
\cite{giles2008} and it decreases with the level, with many samples on
the coarse levels and only a few on the fine levels.  As $\beta >
\gamma$ (cf.~Theorem \ref{th:MLMC_convergence}), we expect the total
cost of the MLMC algorithm $C_{\text{tot}}$ to be proportional to
$\varepsilon^{-2}$. Figure \ref{fig:mlmc_convergence} (right) shows
$\varepsilon^2C_{\text{tot}}$ versus $\varepsilon$, and indeed we
observe a near constant $\varepsilon^2C_{\text{tot}}$ for the MLMC
algorithm across multiple choices of $\varepsilon$. Figure
\ref{fig:mlmc_convergence} also compares the MLMC cost with the cost
of obtaining an estimate of the same accuracy with standard Monte
Carlo. We observe that the MLMC algorithm offers significant
computational savings compared to standard Monte Carlo, with an
improvement in the total cost $C_{\text{tot}}$ of up to $3$ orders of
magnitude (Figure \ref{fig:mlmc_convergence}, right).
\begin{figure}[h!]
	\centering
	\includegraphics[trim=3cm 0cm 4cm 1cm, clip=true, width=.8\textwidth]{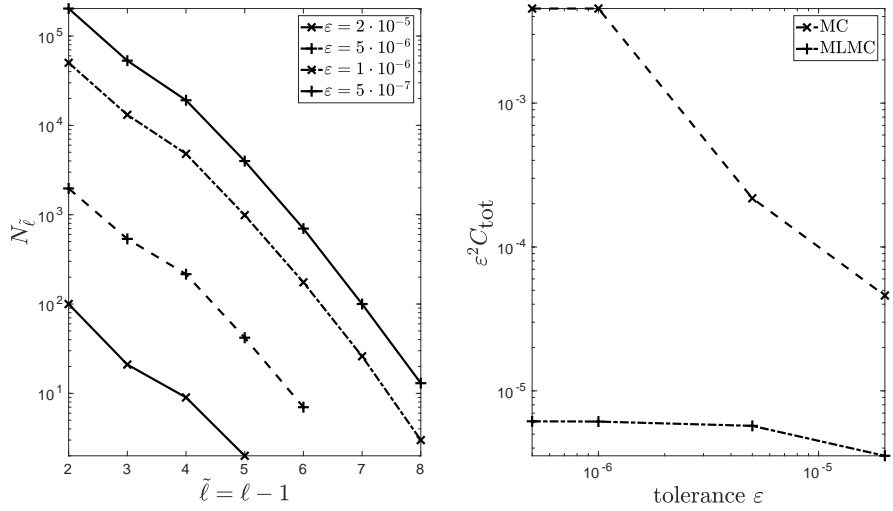}
	\centering
	\caption{MLMC convergence for the solution of \eqref{eq:PDE_of_interest}. In the plot on the left we show how the MLMC algorithm automatically selects the optimal number of samples $N_{\tilde{\ell}}$ on each level to achieve a given tolerance $\varepsilon$. Note that the MLMC routine uses the second mesh in the hierarchy described in Table \ref{tab:table1} to define the first level $\tilde{\ell}$. The first mesh in Table \ref{tab:table1} is dropped since it is too coarse and it would not bring any significant advantage to the performance of MLMC \cite{giles2015multilevel}. In the plot on the right we compare the efficiency of MLMC with standard MC for different tolerances. The savings of MLMC with respect to standard Monte Carlo are considerable.}
	\label{fig:mlmc_convergence}
\end{figure}

Finally, we consider the convergence of MLMC with $p$-refinement. We follow the same procedure as in the $h$-refinement case and solve \eqref{eq:PDE_of_interest} after approximating the coefficient $u$ by solving \eqref{eq:white_noise_eqn} with FEM. This time, however, we fix the mesh to be the coarsest mesh in the 2D hierarchy (cf. Table \ref{tab:table1}) and we consider a hierarchy of continuous piecewise polynomial elements of increasing polynomial degree $p=1,\dots,8$. We investigate the convergence behavior for different values of $\nu$, namely $\nu\in\{1,7,31\}$ (corresponding to $k\in\{1,3,15\}$).

We observe in Figure \ref{fig:p-refinement} that convergence is geometric (the error decreases exponentially as the polynomial degree $p$ grows). The solution of \eqref{eq:PDE_of_interest} is actually almost surely not analytic and we would therefore expect algebraic convergence (i.e.~the error decreases polynomially as $p$ grows) \cite{babuska1994}. We hypothesize that this better-than-expected convergence is in fact pre-asymptotic behavior and that the geometric convergence will eventually plateau and switch to a slower algebraic rate that depends on the smoothness of $u$ (the larger $\nu$, the faster the convergence) \cite{Hackbusch1992,Schwab1998}. However, apart from the $\nu=1$ case for which the convergence plot begins to tail off, this is not observed for the polynomial degrees considered. We note that the larger the smoothness parameter $\nu$ is, the faster the expected value converges. The variance convergence order, on the other hand, seems to be unaffected by the value of $\nu$.

\begin{figure}[h!]
	\centering
	\includegraphics[trim=3cm 0cm 4cm 1cm, clip=true, width=.8\textwidth]{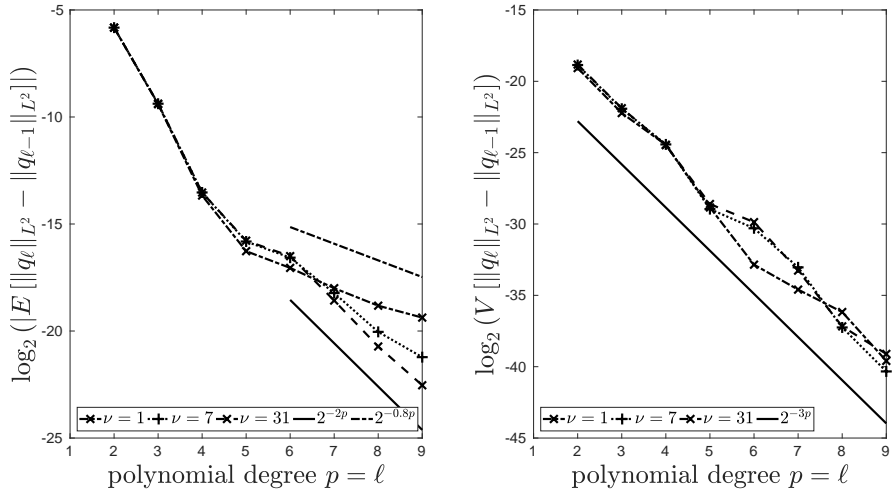}
	\centering
	\caption{Convergence behavior of the FEM approximation to the solution of \eqref{eq:PDE_of_interest} with $p$-refinement in 2D. The approximate FEM solution $q_\ell$ on level $\ell$ is obtained by using Lagrange elements of degree $p=\ell$. For the polynomial degrees considered we are only able to observe a pre-asymptotic behavior in which the convergence is geometric. The straight lines (dashed and full) in the left plot indicate the estimated convergence order of the expected value (for $\nu=1$ and $\nu=31$ respectively). The straight line in the right plot indicates the estimated convergence order of the variance for all the values of $\nu$ considered.}
	\label{fig:p-refinement}
\end{figure}

\section{Conclusions}
\label{sec:conclusions}

In this work, we have presented a new sampling technique for efficient
computation of the action of white noise realizations, even when
coupled samples are required within an MLMC framework. This technique
applies for general $L^2$-conforming finite element spaces, and allows
the coupling of samples between non-nested meshes without resorting to
a computationally costly interpolation or projection step. The
numerical results show that our technique works well in practice: the
convergence orders observed agree with existing theory, the number of
supermesh elements grows linearly with the finer parent mesh size, the
covariance structure of the sampled fields converges to the exact
Mat\'ern covariance and the consistency of the telescoping sum is
respected. We note as a concluding remark that our sampling technique
is not limited to Mat\'ern field sampling, but extends naturally to
any application in which spatial white noise realizations are needed
within a finite element framework.

\section*{Acknowledgments}
The authors would like to acknowledge useful discussions with Endre S\"uli, Andrew J. Wathen, Abdul-Lateef Haji-Ali and Alberto Paganini. The authors would also like to express their thanks to James R. Maddison for his assistance with the implementation of an interface between libsupermesh and FEniCS.

\printbibliography

\end{document}